\documentclass[12pt,oneside,a4paper,reqno]{amsart}
\pdfoutput=1

\usepackage[utf8]{inputenc}
\usepackage[T1]{fontenc}
\usepackage{lmodern}
\usepackage{microtype}

\usepackage[margin=2.5cm]{geometry}
\usepackage[foot]{amsaddr}
\usepackage{mathtools}
\usepackage{amssymb}
\usepackage{aligned-overset}
\usepackage{bbm}
\usepackage{upgreek}

\usepackage[sortcites,giveninits=true,eprint=false,url=false,isbn=false,maxnames=99]{biblatex}
\usepackage{amsthm}
\usepackage{enumitem}
\usepackage{url}
\usepackage{hyperref}
\hypersetup{breaklinks=true,colorlinks,allcolors=blue,bookmarksopen,
            pdftitle={Synchronization by noise for traveling pulses},
            pdfauthor={Christian Kuehn, Joris van Winden}}
\usepackage[capitalise]{cleveref} % MUST BE LOADED LAST!
\numberwithin{equation}{section}
\crefname{equation}{}{}
\crefname{assumption}{Assumption}{Assumptions}

\newtheorem{theorem}{Theorem}[section]
\newtheorem{ttheorem}{Theorem}
\newtheorem{lemma}[theorem]{Lemma}
\newtheorem{corollary}[theorem]{Corollary}
\newtheorem{assumption}{Assumption}
\newtheorem{proposition}[theorem]{Proposition}
\theoremstyle{definition}
\newtheorem{definition}[theorem]{Definition}
\theoremstyle{remark}
\newtheorem{remark}[theorem]{Remark}

\newcommand{\cF}{{\mathcal F}}

\newcommand{\cH}{{\mathcal H}}
\newcommand{\cL}{{\mathcal L}}

\newcommand{\cO}{{\mathcal O}}

\newcommand{\cT}{{\mathcal T}}

\newcommand{\cX}{{\mathcal X}}

\newcommand{\bbF}{{\mathbb F}}
\newcommand{\bbN}{{\mathbb N}}
\newcommand{\bbP}{{\mathbb P}}
\newcommand{\bbR}{{\mathbb R}}
\newcommand{\bbT}{{\mathbb T}}
\newcommand{\bbZ}{{\mathbb Z}}

\newcommand{\rd}{{\, \mathrm{d}}}
\newcommand{\diff}{\mathrm{d}}

\newcommand{\eps}{\varepsilon}
\newcommand{\ind}{\mathbbm{1}}
\newcommand{\iso}{\uppi}

\DeclareMathOperator{\Span}{span}

\DeclarePairedDelimiter{\abs}{\lvert}{\rvert}
\DeclarePairedDelimiter{\norm}{\lVert}{\rVert}
\DeclarePairedDelimiter{\cur}{\{}{\}}
\DeclarePairedDelimiter{\bra}{(}{)}
\DeclarePairedDelimiter{\sqr}{[}{]}

\newcommand{\PP}[2][]{\mathbb{P}\, \sqr[#1]{#2}}
\newcommand{\EE}[2][\big]{\mathbb{E} \sqr[#1]{#2}}

\title{Synchronization by noise for traveling pulses}

\author{Christian Kuehn}
\address{Technical University of Munich, School of Computation Information and Technology, Department of Mathematics,  
         Boltzmannstrasse 3, 85748 Garching, 
         Germany.}
\email{ckuehn@ma.tum.de}

\author{Joris van Winden}
\address{Delft University of Technology,
         Faculty of Electrical Engineering, Mathematics and Computer Science,
         Mekelweg 4, 2628 CD Delft, 
         Netherlands}
\email{J.vanWinden@tudelft.nl}

\thanks{JvW is supported by a DIAM fast-track scholarship. CK would like to thank the VolkswagenStiftung for support via a Lichtenberg Professorship.}


\date{\today}
\begin{document}

\begin{abstract}
    We consider synchronization by noise for stochastic partial differential equations which support traveling pulse solutions, such as the FitzHugh--Nagumo equation.
    We show that any two pulse-like solutions which start from different positions but are forced by the same realization of a multiplicative noise, converge to each other in probability on a time scale $\sigma^{-2} \ll t \ll \exp(\sigma^{-2})$, where $\sigma$ is the noise amplitude.
    The noise is assumed to be Gaussian, white in time, colored and periodic in space, and non-degenerate only in the lowest Fourier mode.
    The proof uses the method of phase reduction, which allows one to describe the dynamics of the stochastic pulse only in terms of its position.
    The position is shown to synchronize building upon existing results, and the validity of the phase reduction allows us to transfer the synchronization back to the full solution.
\end{abstract}

\maketitle
    
\section{Introduction}
\label{sec:intro}

We study SPDEs with traveling pulse-like solutions, such as the stochastic FitzHugh--Nagumo equation:
\begin{equation}
    \label{eq:fhn}
    \begin{aligned}
    \partial_t u &= \nu \partial_{xx}u + u(u-a)(1-u) - v + \sigma g(u,v) \circ \partial_t W, \\
    \partial_t v &= \eps(u - \gamma v),
    \end{aligned}
\end{equation}
where $\nu > 0$, $0 < a < 1/2$, $\eps \ll 1$, $W$ is a suitable Gaussian noise and $g \colon \bbR^2 \to \bbR$ is smooth.
The parameter $\sigma > 0$ controls the amplitude of the noise, and we are especially interested in the small noise regime $\sigma \ll 1$.
Being a prototypical model for neural pulse propagation, it is well known that \eqref{eq:fhn} with $\sigma = 0$ admits a stable traveling pulse solution, moving with a fixed speed $c > 0$; see e.g.~\cite{jones_stability_1984,conley_application_1984,guckenheimer_homoclinic_2010}.
In particular, the position of the pulse at time $t_0$ can be retroactively determined from the position at any later time $t_1$ simply by subtracting $c(t_1 - t_0)$.

In the presence of noise ($\sigma > 0$), the situation is significantly different.
Although pulse-like solutions persist \cite{hamster_stability_2020,eichinger_multiscale_2022}, the noise causes a random shift of the pulse position, and in many cases this results in a change in velocity \cite{hamster_stability_2020} (see also \cite{hamster_stability_2019,mueller_random_1995,westdorp_longtimescale_2024} for results on other equations).
Moreover, in sharp contrast to the deterministic case, it has been observed that the pulse position at large times is nearly independent of the initial position \cite{kilpatrick_stochastic_2015,wang_coherence_2000,teramae_noise_2006,nakao_phasereduction_2014,pikovsky_synchronization_2001}.
Instead, the large-time position is mainly determined by stochastic forcing, a phenomenon referred to as \emph{synchronization by noise}.
However, despite strong evidence for synchronization of pulses, no mathematical proof has been given until now.

Our main result (\cref{thm:spdesync}) establishes the first rigorous proof of synchronization by noise for traveling pulse solutions to \eqref{eq:fhn}.
Specifically, we prove that any two pulse-like solutions, starting at different positions and forced by the same noise, converge to each other in probability for $t \gg \sigma^{-2}$ as $\sigma \searrow 0$. We mention already that this theoretical synchronization result can also have concrete implications for biophysical systems such as memory processes in neuroscience~\cite{fell_role_2011} and cardiac arrhythmia~\cite{luther_lowenergy_2011}. In both applications, the FitzHugh--Nagumo equation~\eqref{eq:fhn} is the baseline model problem studied. We shall return to applied aspects in~\cref{sec:outlook} as the proof of our main result is also very insightful from a practical perspective. 

\subsection{Main result}
Our result is formulated for an abstract semilinear PDE of the form
\begin{equation}
    \label{eq:pde}
    \diff u = A u \rd t + f(u) \rd t,
\end{equation}
where $u \colon \bbR^+ \times \bbR \ni (t,x) \mapsto u(t,x) \in \bbR^n$ is continuous, $f \colon \bbR^n \to \bbR^n$ is sufficiently smooth, and $A$ should be thought of as a (possibly degenerate) differential operator with constant coefficients.

We briefly describe our assumptions, which are formulated in more detail in \cref{subsec:assanal,subsec:assprob}.
The first two assumptions allow for a robust solution theory (\cref{ass:analytic}) and ensure the existence of a \emph{pulse profile} $u^*$ and \emph{pulse speed} $c \in \bbR$ such that $\hat{u}(t,x) \coloneq u^*(x - ct)$ is an orbitally stable traveling pulse solution to \eqref{eq:pde} (\cref{ass:pulse}).
It is well known that the FitzHugh--Nagumo equation \eqref{eq:fhn} (with $\sigma = 0$) fits into this setting.
For the stochastic equation, we introduce the noise amplitude $\sigma > 0$ and write
\begin{equation}
    \label{eq:spde}
    \diff u_{\sigma} = A u_{\sigma} \rd t + f(u_{\sigma}) \rd t + \sigma g(u_{\sigma}) \circ \diff W(t).
\end{equation}
We take $g \colon \bbR^n \to \bbR^n$ sufficiently smooth and let $W(t,x)$ be a scalar\footnote{There is no obstacle to considering vector-valued noise except for significant notational inconvenience.} Gaussian noise which is white in time, colored and periodic in space, and weakly non-degenerate (\cref{ass:noise}).

The nondegeneracy condition on $W$ is formulated in terms of $g$, $u^*$, the lowest Fourier modes of $W$, and an adjoint eigenfunction $\psi$ which can be calculated from \eqref{eq:pde} and the pulse profile $u^*$.
The symbol `$\circ$' in \eqref{eq:spde} indicates that the stochastic integral is interpreted in the Stratonovich sense, but this does not play a major role in the proof.

\cref{ass:analytic,ass:pulse,ass:noise} taken together already suffice to show synchronization for the reduced SDE that describes the pulse position.
However, we are only able to transfer this synchronization back to the `full' equation \eqref{eq:spde} by additionally assuming that \emph{either} the pulse speed $c$ is zero, \emph{or} that the noise $W$ has spatially homogeneous statistics (\cref{ass:scaling}).
We refer ahead to \cref{subsec:liftscaling} for further discussion of this assumption.

Let us now state the main result, which holds under \cref{ass:analytic,ass:pulse,ass:noise,ass:scaling}.
For $x \in \bbR$ and $\sigma > 0$, we let $u^x_\sigma(\cdot)$ denote the (mild) solution of \eqref{eq:spde} with initial condition $u^x_{\sigma}(0,\cdot) = u^*(\cdot - x)$. 
\begin{ttheorem}
    \label{thm:spdesync}
    Suppose that the times $(t_{\sigma})_{\sigma > 0}$ satisfy
    \begin{equation}
        \label{eq:tsigmacond}
        \lim_{\sigma \to 0} t_{\sigma}\sigma^2 = \infty,
        \qquad 
        0 \leq t_{\sigma} \leq \exp(\sigma^{-2+q})
    \end{equation}
    for some $q \in (0,2)$.
    Then for any $x,y \in \bbR$, we have
    \begin{equation}
        \label{eq:spdesync}
        \inf_{n \in \bbZ}\bra[\big]{ \norm{u^{x}_{\sigma}(t_{\sigma},\cdot + n) - u^{y}_{\sigma}(t_{\sigma},\cdot)}_\cX} \overset{\bbP}{\to} 0
    \end{equation}
    as $\sigma \to 0$.
\end{ttheorem}
In other words, \emph{any} two solutions to \eqref{eq:spde} which start in the set of translates $\cur{u^*(\cdot - x) : x \in \bbR}$ synchronize (modulo integer translations) on the time scale $\sigma^{-2} \ll t_{\sigma} \ll \exp(\sigma^{-2})$.
\begin{remark}
   The function space $\cX$ is defined in \cref{ass:analytic}, and always embeds into the space of continuous functions by assumption. 
\end{remark}
\begin{remark}
    The infimum over $n$ in \eqref{eq:spdesync} cannot be removed, since the periodic noise guarantees that $u^{x+n}_\sigma(t,z+n) = u^x_\sigma(t,z)$ always holds for $n \in \bbN$.
    Our results can be straightforwardly adapted to traveling pulses on periodic domains, in which case the infimum can be omitted.
    However, for the sake of readability and in view of applications we have chosen the setting where \eqref{eq:spde} is posed on the real line.
\end{remark}

\subsection{Synchronization by noise}
\label{subsec:introsync}
There has been a continued mathematical interest in synchronization by noise, starting with pioneering results on stochastic flows of diffeomorphisms by \textcite{baxendale_asymptotic_1986,martinelli_small_1988}, and continuing with seminal works by \textcite{crauel_additive_1998,arnold_random_1998}.
More recently, synchronization for SPDEs has been considered, with a particular focus on the Chafee--Infante (or Allen--Cahn) equation \cite{bianchi_additive_2016,caraballo_effect_2006,blumenthal_pitchfork_2023}.
Abstract criteria for synchronization have been given in \cite{flandoli_synchronization_2017,flandoli_synchronization_2017a,newman_necessary_2018}, which has led to synchronization results for porous media equations \cite{flandoli_synchronization_2017a} and gradient type S(P)DEs \cite{flandoli_synchronization_2017,gess_lyapunov_2024}, all with additive and highly non-degenerate noise.

A common thread in these works is that there are two main strategies for showing synchronization\footnote{A notable exception is \cite{bianchi_additive_2016}, which instead relies on the noise being constant in space.}:
\begin{itemize}
    \item Exploit an \emph{order-preserving} structure \cite{flandoli_synchronization_2017a,chueshov_structure_2004,crauel_additive_1998,arnold_orderpreserving_1998,martinelli_small_1988,caraballo_effect_2006}.
    \item Combine \emph{asymptotic stability} and \emph{irreducibility} properties \cite{gess_lyapunov_2024,flandoli_synchronization_2017,newman_necessary_2018,baxendale_statistical_1991}.
\end{itemize}
Currently, it seems that showing synchronization for \eqref{eq:fhn} by attacking the equation head-on is totally out of reach with both methods.
Firstly, since \eqref{eq:fhn} is a two-component SPDE, we deem it unlikely that a useful order-preserving structure exists.
Secondly, showing asymptotic stability generally requires one to verify negativity of the top Lyapunov exponent.
Even with detailed information of the invariant measure (which is unavailable for \eqref{eq:fhn}) this is a challenging task, see e.g.\ \cite{gess_lyapunov_2024}.
Thirdly, to have sufficient irreducibility properties, it is typically needed to impose strict nondegeneracy conditions on the noise.
In contrast, we require nondegeneracy only in the lowest Fourier mode (see \cref{ass:noise}), a setting in which mixing (typically a necessary condition for synchronization) is already hard to prove \cite{hairer_ergodicity_2006}.

We completely bypass these difficulties using the method of \emph{phase reduction}.
The key insight is that the dynamics of a pulse-like solution can be accurately described by tracking only the position of the pulse, and forgetting all other information.
To show that the pulse synchronizes, it thus suffices to show that the position synchronizes.

In the physics literature, the possibility of showing synchronization by noise through phase reduction has been extensively discussed, especially for limit-cycle oscillators \cite{nakao_phase_2016,nakao_synchrony_2005,nakao_noiseinduced_2007,nakao_effective_2010,pikovskii_synchronization_1984,pikovsky_phase_1997,pikovsky_phase_2000,teramae_robustness_2004,teramae_noise_2006,wang_coherence_2000}.
However, the (temporal) roughness of the noise makes the phase reduction into quite a delicate procedure, which is seen in \cite{yoshimura_phase_2008} and in \cref{rem:wrongreduction} ahead.
This further motivates the importance of a rigorous proof, since subtle errors in the phase reduction can immediately invalidate a synchronization result.

\subsection{Phase reduction}
\label{subsec:phasereduction}
Consider a solution $u_{\sigma}$ to \eqref{eq:spde}, which has the deterministically stable pulse profile $u^*$ as its initial condition.
When $\sigma \ll 1$ the solution $u_{\sigma}(t)$ will still resemble a translate of $u^*$ with high probability, as long as $t \ll \exp(\sigma^{-2})$.
Thus, it is possible to associate to $u_{\sigma}$ an auxiliary scalar process $\gamma_{\sigma}$, henceforth referred to as the \emph{phase}, such that we have
\begin{equation}
    \label{eq:introphasereduce}
    u_{\sigma}(t,x) = u^*(x - \gamma_{\sigma}(t)) + \cO(\sigma).
\end{equation}
The infinite-dimensional dynamics of $u_{\sigma}$ then effectively reduce to the one-dimensional dynamics of $\gamma_{\sigma}$, and this will be our main tool to show synchronization.

However, the condition \eqref{eq:introphasereduce} only characterizes $\gamma_{\sigma}$ up to $\cO(\sigma)$, and therefore does not lead to a canonical definition of the phase process.
This is reflected in the recent literature on stochastic traveling waves/pulses, in which a variety of (non-equivalent) phase tracking methods, each with their own strengths and weaknesses, have been proposed and developed \cite{kruger_front_2014,hamster_stability_2019,inglis_general_2016,vanwinden_noncommutative_2024,adams_existence_2024}.
We will not give an overview of the different methods, but we emphasize that whichever method is preferable typically depends on the result one is trying to prove.

The phase tracking method of choice in this paper is the \emph{isochronal phase}, which originates in the context of nonlinear oscillators \cite{winfree_patterns_1974}. It was put on a mathematical footing by \textcite{guckenheimer_isochrons_1975} and has been applied to the study of transient patterns in \cite{adams_existence_2024,adams_isochronal_2025}.
Briefly, the idea is as follows. Fix a profile $v(x)$, and let $u(t,x)$ be the solution to \eqref{eq:pde} with initial condition $u(0,x) = v(x)$.
Assuming that the profile $v(\cdot)$ resembles a translate of the pulse profile $u^*$, it follows from the deterministic stability theory that there exists $\iso(v) \in \bbR$ such that
\begin{equation*}
    \lim_{t \to \infty} \norm{u(t,\cdot) - u^*(\cdot - ct - \iso(v))} = 0.
\end{equation*}
Moreover, $\iso(v)$ is uniquely determined by $v$ (justifying the notation), and the approximation $v \approx u^*(\cdot - \iso(v))$ is valid.
The map $\iso$ is referred to as the \emph{isochron map}.
Note that $\iso$ is constructed from the dynamics of \eqref{eq:pde}, and thus is a purely deterministic mapping from the state space of \eqref{eq:pde} into $\bbR$.

To track the position of the stochastic pulse $u_{\sigma}$, we may now define the \emph{isochronal phase}\footnote{This notion of isochronicity differs from the one introduced in \cite{engel_random_2021}, where a (time-dependent and anticipating) random isochron map is constructed directly from the stochastic dynamics.}
process simply by setting $\gamma_{\sigma} \coloneq \iso(u_{\sigma})$.
To obtain a detailed description of the phase, we apply the It\^o formula (which is shown to hold in \cite{adams_existence_2024}) to $\iso(u_{\sigma})$ and use \eqref{eq:introphasereduce} to simplify the resulting expression.
This results in an (approximate) SDE for $\gamma_{\sigma}$ which no longer makes any reference to $u_{\sigma}$.
In view of~\eqref{eq:introphasereduce} the dynamics of $u_{\sigma}$ are then characterized by this SDE, and the phase reduction is complete.
We refer ahead to \cref{subsec:dothereduction} for the full derivation of the phase-reduced SDE, which is given by \eqref{eq:sdegamma}.

\subsection{Proof strategy}
Phase reduction can be a powerful tool to show synchronization.
We demonstrate this by means of a (surprisingly robust) analogy with the well-studied double-well SDE, given by
\begin{equation}
    \label{eq:introtoy}
    \diff X = (X - \abs{X}^2 X) \rd t + \sigma \rd W(t), \qquad X \in \bbR^2,
\end{equation}
where $W = (\beta_1, \beta_2)$ is a two-dimensional Brownian motion and $\sigma > 0$.
We first observe that the set $S = \cur{X : \abs{X} = 1}$ forms a connected set of stable equilibria for the deterministic equation.
Thus, it is expected that the radial component of a typical trajectory of \eqref{eq:introtoy} will satisfy $\abs{X} = 1 + \mathcal{O}(\sigma)$.
On the other hand, the angular component of $X$ will diffusively wander on a one-dimensional torus.
Indeed, by transforming to polar coordinates via $(X_1,X_2) = (R \cos(\Phi), R \sin(\Phi))$, we find \cite{vorkastner_approaching_2020}:
\begin{equation*}
    \diff \Phi = \sigma R^{-1}(-\sin(\Phi) \rd \beta_1(t) + \cos(\Phi) \rd \beta_2(t)).
\end{equation*}
Since we expect that $R \approx 1$ when $\sigma$ is small, we can safely neglect the dynamics of $R$ to reduce the dynamics of $X$ to those of the angular coordinate $\Phi$, which satisfies
\begin{equation}
    \label{eq:introtoyphase}
    \diff \Phi \approx \sigma (-\sin(\Phi) \rd \beta_1(t) + \cos(\Phi) \rd \beta_2(t)).
\end{equation}
This SDE is known to synchronize \cite{baxendale_statistical_1991}, and by rescaling time we see that the synchronization occurs on the time scale $t \sim \sigma^{-2}$.
Moreover, one can check that \eqref{eq:introtoyphase} is valid on the time scale $t \sim \sigma^{-2}\log(\sigma^{-1})$.
Hence, it is expected that we can transfer the synchronization of \eqref{eq:introtoyphase} back to \eqref{eq:introtoy} to conclude that $X$ synchronizes for $t \gg \sigma^{-2}$ in the limit $\sigma \searrow 0$.
We remark that it is crucial that the reduced dynamics accurately capture the full dynamics on a time scale which is \emph{strictly longer} than the typical time until synchronization.
Already for \eqref{eq:introtoy} this is delicate, as the two time scales differ only by a logarithmic factor.

The argument sketched above consists of three main steps.
\begin{enumerate}[label=(\roman*)]
    \item \label{it:step:1} Reduce the dynamics by describing \eqref{eq:introtoy} in terms of \eqref{eq:introtoyphase}.
    \item \label{it:step:2} Show synchronization of the reduced dynamics.
    \item \label{it:step:3} Transfer synchronization back to the full dynamics.
\end{enumerate}
In our actual proof, \ref{it:step:1} is achieved using the isochronal phase as outlined in \cref{subsec:phasereduction}.
Although the resulting SDE is not as explicit as \eqref{eq:introtoyphase}, we still obtain a detailed description of the coefficients by exploiting symmetries of the isochron map $\iso$ which are inherited from \eqref{eq:pde}.
This allows us to carry out \ref{it:step:2} by combining abstract results by \textcite{flandoli_synchronization_2017} with a calculation of the Lyapunov exponent and a control-theoretic argument.
The time scales of approximation and synchronization are the same as for \eqref{eq:introtoyphase}, which finally allows \ref{it:step:3} to succeed.

\subsection{Outline}
In \cref{sec:setting} we provide notation, specify our setting and assumptions, and formulate preliminaries regarding random dynamical systems.
In \cref{sec:phasereduction} we derive the approximate SDE describing the pulse position and prove error estimates.
We simplify and analyze the phase-reduced SDE in \cref{sec:reducedsdeanalysis}.
We show that the reduced SDE synchronizes sufficiently fast, which allows us to transfer the synchronization behavior back to the full SPDE to prove \cref{thm:spdesync}.

\section{Setting, assumptions and preliminaries}
\label{sec:setting}
\subsection{Notation}
\label{subsec:notation}
We use the convention that $\bbN$ does not include zero, and write $\bbN_0 = \bbN \cup \cur{0}$, as well as $\bbR^+ = [0,\infty)$.
We write $\bbT = \bbR /  \bbZ$ to denote the (flat) one-dimensional torus with unit length. 

We write $\norm{\cdot}_\cX$ for the norm of a general Banach space $\cX$.
When $(E,d)$ is a metric space and $x \in E$, $r > 0$, we write $B_r(x)$ for the open ball around $x$ of radius $r$.
We write $C_b(E)$ for the Banach space of bounded continuous functions $f \colon E \to \bbR$, equipped with the supremum norm.
When $k \in \bbN$ and $A$ is either $\bbR$, $\bbT$, or an open subset of a Banach space, we write $C^k(A;\cX)$ for the Banach space of functions $f \colon A \to \cX$ which have $k$ bounded continuous (Fr\'echet) derivatives, equipped with the usual norm.
In the case $\cX = \bbR$, we simply write $C^k(A)$.
We also write $\ell^2(\bbZ;\cX)$ for the space of (norm) square-integrable functions $f \colon \bbZ \to \cX$. If $\cX = \bbR$, we write $\ell^2(\bbZ)$ or even just $\ell^2$.
We write $H^{s,p}(\bbR;\bbR^n)$ for the Bessel potential space of functions from $\bbR$ to $\bbR^n$ with smoothness $s$ and integrability $p \in (1,\infty)$.
We write $\EE[]{\cdot}$ for the expectation associated with the probability space $(\Omega,\cF,\bbP)$ which is specified in \cref{subsec:probrds}.
The notation $A \xrightarrow{\bbP} B$ indicates convergence in probability with respect to $\bbP$.

We use the notation $A \lesssim B$ to mean that there exists a constant $C$, possibly depending on the objects introduced in \cref{ass:analytic,ass:pulse,ass:noise}, such that $A \leq CB$.
In the statements and proofs of \cref{thm:longtermstab,thm:longtermapprox}, we will additionally use the notation $\sigma \ll_q 1$ to mean that $\sigma$ is sufficiently small, based (only) on $q$ and the objects from \cref{ass:analytic,ass:pulse,ass:noise}.

\subsection{Analytic setting and linear stability}
\label{subsec:assanal}
We begin by fixing the state space for \eqref{eq:spde} by letting $p \in [2,\infty)$, $s > 1/p$, $n \in \bbN$ and setting $\cX \coloneq H^{s,p}(\bbR;\bbR^n)$.
The process $u$ is intended to be continuous in time with values in $\cX$.
\begin{remark}
    \label{rem:analytic}
    The condition on $p$ is used in the proof of \cref{thm:longtermstab}, which requires Gaussian tail estimates for $\cX$-valued stochastic integrals.
    Such tail estimates are available if $\cX$ is $2$-smooth \cite{seidler_exponential_2010}, a property which holds for Bessel spaces iff $p \in [2,\infty)$ (this follows e.g.\ from \cite[Proposition 2.1]{pinelis_optimum_1994} and the fact that $L^p$ is isometrically isomorphic to $H^{s,p}$).
    The condition $s > 1/p$ is necessary for $\cX$ to be a Banach algebra, a fact which is used at several points in the argument.
    By the Sobolev embedding, this restricts us to work exclusively with continuous functions.
\end{remark}

To have a robust solution theory for \eqref{eq:pde}, we make the following assumption:
\begin{assumption}[PDE setting]
    \label{ass:analytic}
    The following conditions hold:
    \begin{itemize}
        \item $A$ generates a $C_0$-semigroup on $\cX$ which commutes with translations.
        \item The Nemitskii map $u \mapsto f(u)$ is seven times Fr\'echet differentiable from $\cX$ to $\cX$.
    \end{itemize}
\end{assumption}

The next assumption guarantees existence of a stable traveling pulse solution to \eqref{eq:pde}.
\begin{assumption}[Stable traveling pulse]
    \label{ass:pulse}
    There exist $c \in \bbR$ and $u^* \in \cX$ such that $u(t,x) = u^*(x - ct)$ solves \eqref{eq:pde} (in the mild sense) and $\partial_x^{(n)} u^* \in \cX$ for $n \in \cur{1,2,3,4}$.
    Moreover, the linear operator $\cL = A + c\partial_x + f'(u^*)$ generates a $C_0$-semigroup $(P(t))_{t \geq 0}$ on $\cX$, and there exist projections $\Pi^c,\Pi^s$ on $\cX$ and $C,a > 0$ such that
    \begin{itemize}
        \item $I_\cX = \Pi^c + \Pi^s$,
        \item $\Pi^c \cX = \Span \cur{\partial_x u^*}$,
        \item $\norm{P(t)\Pi^s f}_{\cX} \leq C e^{-at}\norm{f}_{\cX}, \quad t \geq 0$, $f \in \cX$.
    \end{itemize}
\end{assumption}
\begin{remark}
    \label{rem:translateustar}
    Since $\partial_x u^* \in \cX$, it follows that we have the estimate
    \begin{equation*}
        \norm{u^*(\cdot - x) - u^*(\cdot - y)}_\cX \lesssim \abs{x - y}, \qquad x,y \in \bbR.
    \end{equation*}
\end{remark}
By duality, we see from \cref{ass:pulse} that there exists a (bounded) functional $\psi \colon \cX \to \bbR$ such that
\begin{equation}
    \label{eq:Picpsi}
    \Pi^c f = -\psi(f) \partial_x u^* \eqcolon -\langle \psi,f\rangle \partial_x u^*, \qquad f \in \cX,
\end{equation}
where the brackets denote the duality pairing.
We prefer to write our equations in terms of $\psi$ rather than $\Pi^c$, and will use the right-hand side of \eqref{eq:Picpsi} whenever possible.
In many cases it holds that $\psi$ is a smooth function and satisfies $\cL^* \psi = 0$, where $\cL^*$ is the formal $L^2$-adjoint of $\cL$.

\subsection{Noise}
\label{subsec:assprob}
In order to have suitable ergodic properties of the phase-reduced SDE, we restrict our noise to be periodic in space.
By rescaling the spatial variable, we may additionally assume that the period is $1$ without loss of generality. This motivates the expansion of the noise in terms of the orthonormal basis $(e_k)_{k \in \bbZ}$ of $L^2(\bbT)$ given by
\begin{equation}
    \label{eq:basis}
    e_k(x) = \begin{cases}
        \sqrt{2}\cos(2\pi k x), & k > 0 , \\
        1, & k = 0, \\
        \sqrt{2}\sin(2\pi k x), & k < 0.
    \end{cases}
\end{equation}
We thus let $\alpha = (\alpha_k)_{k \in \bbZ}$ be a sequence of coefficients and let the noise be given by:
\begin{equation}
    \label{eq:defnoise}
    W(t,x) = \sum_{k \in \bbZ} \alpha_k e_k(x) \beta_k(t),
\end{equation}
where $(\beta_k)_{k \in \bbZ}$ is a sequence of independent Brownian motions (see \cref{subsec:probrds} for the structure of our probability space).
Even though \eqref{eq:defnoise} is not the most general noise possible, it is a flexible formulation which is commonplace in the literature on SPDEs (see e.g. \cite{hairer_ergodicity_2006,galeati_convergence_2020,fischer_existence_2018}).
Furthermore, it has the benefit of allowing us to formulate the coming assumptions in terms of only the coefficients $(\alpha_k)_{k \in \bbZ}$.

\begin{assumption}[Noise regularity and nondegeneracy]
    \label{ass:noise}
    Let $\nu \coloneq s - 1/p$ (recall $\cX = H^{s,p}(\bbR;\bbR^n)$).
    The Nemitskii map $u \mapsto g(u)$ is four times Fr\'echet differentiable from $\cX$ to $\cX$, and the noise coefficients satisfy
    \begin{equation}
        \label{eq:noisereg}
        \sum_{k \in \bbZ} \abs{k}^{2\nu} \alpha_k^2  < \infty.
    \end{equation}
    Furthermore, both $\alpha_{-1}$ and $\alpha_1$ are nonzero and there exists $x \in \bbR$ such that
    \begin{equation}
        \label{eq:noisenondegen}
        \langle \psi g(u^*), \sin(2 \pi\cdot + x) \rangle \neq 0,
    \end{equation}
    where $u^*,\psi$ are as in \cref{ass:pulse} and \eqref{eq:Picpsi}.
\end{assumption}
\begin{remark}
    \label{rem:noisereg}
    Since $\nu > 0$ by the conditions on $s$ and $p$, it follows that \eqref{eq:defnoise} converges almost surely in the H\"older space $C^{\alpha}$ for $\alpha < \nu$, which in particular forces the noise to be continuous in space.
    Since our state space $\cX$ always embeds into the space of continuous functions (see \cref{rem:analytic}) and we do not assume any smoothing properties, the condition that $(\alpha_k)_{k \in \bbZ} \in \ell^2(\bbZ)$ thus forms a natural barrier for the noise regularity.
    In our intended applications we can however take $\nu$ arbitrarily close to zero.
\end{remark}
Under \cref{ass:analytic,ass:noise}, it follows from well-known arguments (see e.g.\ \cite{daprato_stochastic_1992,brzezniak_stochastic_1995}) that \eqref{eq:spde} has a unique (local) mild solution.
As $\cX$ is a Banach algebra and $\norm{e_k}_\cX \lesssim \abs{k}^{\nu}$, we have
\begin{equation}
    \sum_{k \in \bbZ} \alpha_k^2 \norm{g(u)e_k}_\cX^2 \lesssim \norm{g(u)}_\cX^2 \sum_{k \in \bbZ} \alpha_k^2 \abs{k}^{2\nu}, \qquad u \in \cX,
\end{equation}
which implies that the relevant stochastic integrals converge in the topology of $\cX$.
Moreover, even though the solution $u_{\sigma}$ could potentially blow up in finite time, the stability of the pulse (see \cref{thm:longtermstab} ahead) guarantees that such a blow-up becomes increasingly unlikely in the small noise limit $\sigma \searrow 0$.
Hence, for our purposes it is harmless to think of the solution $u_{\sigma}$ as existing for all times.

We now state our final assumption.
We have formulated it separately from \cref{ass:noise}, since it is only needed once near the end of the proof, where it is used to transfer synchronization results from the phase-reduced SDE back to the SPDE.
\begin{assumption}
    \label{ass:scaling}
    At least one of the following conditions hold:
    \begin{itemize}
        \item $c = 0$,
        \item $\alpha_{k} = \alpha_{-k}$ for every $k \in \bbN$.
    \end{itemize}
\end{assumption}
Note that the first condition in \cref{ass:scaling} can always be satisfied by changing to a co-moving coordinate frame (in which $c = 0$) before adding noise of the form \eqref{eq:defnoise}.
The alternative condition is equivalent to the statistics of the noise being spatially homogeneous.
For a detailed discussion of the implications of \cref{ass:scaling} and the possibility to lift it, we refer ahead to \cref{subsec:liftscaling}.

\subsection{Probability, RDS, synchronization}
\label{subsec:probrds}
We now provide the technical setup and preliminaries for our use of random dynamical systems (RDS) in \cref{sec:reducedsdeanalysis}.
We will keep the RDS theory to a minimum by introducing only the concepts which we directly use.
Note that the stated definitions are standard, and robustly generalize to more involved settings.
For detailed expositions, we refer the reader to \cite{arnold_random_1998,daprato_ergodicity_1996}.

In our setting, the only `source' of randomness is $W(t)$, which is constructed from independent one-dimensional Brownian motions by \eqref{eq:defnoise}.
Hence, to construct our probability space we let $\Omega = C(\bbR \times \bbZ)$, let $\cF$ be the Borel $\sigma$-field of the compact-open topology on $\Omega$, and let $\bbP$ be the probability measure on $(\Omega,\cF)$ such that the random variables $\beta_k(t) \coloneq \omega(t,k)$ form a sequence of independent two-sided Brownian motions satisfying $\beta_k(0) = 0$.
We also define the family of $\sigma$-fields $\bbF = (\cF_{s,t})_{-\infty\leq s \leq t \leq \infty}$ by taking $\cF_{s,t}$ to be the $\sigma$-field generated by the increments $\cur{ \omega(u,k) - \omega(r,k) :  k \in \bbZ,\, s \leq r \leq u \leq t  }$.
We abbreviate $\cF_t = \cF_{-\infty, t}$. 
Notice that $\cF_{r,s}$ (and thus, $\cF_s$) is independent of $\cF_{t,u}$ when $s \leq t$.

In order to `shift time' in the probability space, we let $\theta = (\theta_t)_{t \in \bbR}$ be the family of transformations of $\Omega$ given by
\begin{equation*}
    \theta_t(\omega)(s,k) = \omega(s + t,k) - \omega(t,k), \qquad t,s \in \bbR,\,k \in \bbZ.
\end{equation*}
Then $\theta_t \circ \theta_s = \theta_{s+t}$ and $\theta_t^*\bbP = \bbP$, so the tuple $(\Omega,\cF,\bbP,\theta)$ forms a (continuous-time) ergodic dynamical system.

Throughout the rest of the section, let $(S,d)$ be a separable complete metric space.
\begin{definition}
    A \emph{perfect cocyle} over $\theta$ is a jointly measurable map $\phi \colon \bbR^+ \times \Omega \times S \to S$, $(t,\omega,x) \mapsto \phi(t,\omega,x)$, which satisfies:
    \begin{itemize}
        \item $\phi(0,\omega,x) = x$ for all $\omega \in \Omega, x \in S$,
        \item $\phi(t+s,\omega,x) = \phi(s, \theta_t \omega, \phi(t,\omega,x)$ for all $t,s \in \bbR^+, \omega \in \Omega, x \in S$.
    \end{itemize}
\end{definition}
The second of these conditions is commonly referred to as the \emph{cocycle property}.
In our setting, the cocycle $\phi(t,\omega,x)$ will be the solution map of an SDE with smooth coefficients driven by the Brownian motions $\beta_k(\cdot) = \omega(\cdot,k)$.
In this case, it holds that $(t,x) \mapsto \phi(t,\omega,x)$ is continuous for all $\omega \in \Omega$.
Furthermore, $\phi(t,\theta_s\cdot,x)$ is $\cF_{s,s+t}$-measurable, and from the independence between the elements of $\bbF$ it follows that $\phi(t,\theta_s\cdot,x)$ is independent of $\cF_{-\infty,s}$.
The tuple $(\Omega,\cF,\bbP,\bbF,\theta,\phi)$ then forms a \emph{continuous white noise random dynamical system}.
Since every element of the tuple except the cocycle has already been fixed, we will only use $\phi$ to refer to the random dynamical system.

When $\phi$ is a continuous white noise RDS, we may define an associated Markovian semigroup $P_t$ via
\begin{equation}
    P_t f(x) = \EE{f(\phi(t,\omega,x))},
\end{equation}
for bounded and measurable $f \colon S \to \bbR$.
We recall that a probability measure $\mu$ on $S$ is called an \emph{invariant measure} for $P_t$ if the identity $\int_S P_t f \rd \mu = \int_S f \rd \mu$ holds for all $t \geq 0$ and all bounded measurable $f \colon S \to \bbR$.
An invariant measure $\mu$ is \emph{ergodic} if for any bounded measurable $f \colon S \to \bbR$, we have $P_t \ind_A = \ind_A$, $\mu$-almost everywhere only if $f$ is constant $\mu$-almost everywhere (several different characterizations exist).
Finally, we say that $\mu$ is \emph{strongly mixing} if $\lim_{t \to \infty} P_t f(x) = \int_S f \rd \mu$ for every $x \in S$ and $f \in C_b(S)$.
This condition is stronger than the usual notion of strong mixing (see \cite[Corollary 3.4.3]{daprato_ergodicity_1996}), but is chosen to remain consistent with \cite{flandoli_synchronization_2017}.

We will use the concept of weak synchronization, introduced in \cite[Definition 2.16]{flandoli_synchronization_2017}.
Recall that a \emph{weak point attractor} for $\phi$ is a random compact set $A(\omega)$ which is invariant (i.e. $A(\theta_t \omega) = \phi(t,\omega,A(\omega))$) and which satisfies
\begin{equation*}
    \lim_{t \to \infty} \PP{d(\phi(t,\omega,x),A(\theta_t \omega)) > \eps} = 0, \qquad x \in S,\, \eps > 0.
\end{equation*}
A weak point attractor is \emph{minimal} if it is contained in any other weak point attractor.
\begin{definition}
    A white noise RDS \emph{synchronizes weakly} if there exists a minimal weak point attractor $a(\omega)$ which consists of a single point $\bbP$-almost surely.
\end{definition}

\section{Phase reduction}
\label{sec:phasereduction}
In the study of stochastic traveling pulses, a key element is the need to define and track the position (henceforth referred to as the \emph{phase}) of the pulse.
Various phase tracking methods have been proposed 
\cite{kruger_front_2014,hamster_stability_2019,inglis_general_2016,vanwinden_noncommutative_2024,adams_existence_2024}, primarily for the purpose of showing stochastic orbital stability of the pulse.
As mentioned in \cref{subsec:phasereduction}, the phase tracking method of choice in this work is the isochronal phase.
This is motivated by several properties which are favorable from a theoretical perspective.
Firstly, unlike phase tracking methods based on minimization or orthogonality conditions, the isochronal phase map is indifferent to the topology of the state space $\cX$ in which \eqref{eq:spde} is solved.
Hence, the phase tracking is decoupled from the solution theory, allowing for flexibility in the choice of $\cX$.
Secondly, the isochronal phase map is defined on the entire basin of attraction of the traveling pulse.
Thus, the phase tracking does not break down for `technical' reasons, and only ceases to function when the solution arguably no longer resembles a traveling pulse anyway.
Thirdly, the isochron map $\iso$ straightforwardly inherits all the symmetry properties of \eqref{eq:pde}, a fact which we rely on to simplify the analysis of the phase-reduced SDE.

\subsection{Isochronal phase}
We now give a brief introduction to the idea behind the isochronal phase and prove some basic properties.
Recall that \cref{ass:pulse} assures that \eqref{eq:pde} has a traveling pulse solution with profile $u^*$.
Hence, for $\eps$ sufficiently small, the set
\begin{equation}
    \label{eq:pulseneighbourhood}
    \Gamma_{\eps} \coloneq \cur{u \in \cX : \inf_{x \in \bbR} \norm{u(\cdot) - u^*(\cdot - x)}_\cX < \eps}
\end{equation}
is contained in the basin of attraction of the traveling pulse solutions.
Indeed, writing $u^v$ for the solution to \eqref{eq:pde} with initial condition $u(0) = v$, we have the following theorem:
\begin{theorem}
    \label{thm:isochron}
    There exists $\delta > 0$, such that for every $v \in \Gamma_{\delta}$ there exists a unique $\iso(v) \in \bbR$ such that
    \begin{equation}
        \label{eq:isochronconv}
        \lim_{t \to \infty}\norm{u^v(t,\cdot) - u^*(\cdot - ct - \iso(v))}_\cX = 0.
    \end{equation}
    Moreover, if $v \in \Gamma_{\delta'}$ with $\delta' \leq \delta$, we have
    \begin{equation}
        \label{eq:isochronacc}
        \norm{\cT_{\iso(v)}u^* - v}_\cX \lesssim \delta'.
    \end{equation}
\end{theorem}
\begin{remark}
    In the language of dynamical systems, the isochronal phase provides a \emph{foliation} of $\cX$ near the \emph{center manifold} $\cur{u^*(\cdot - x) : x \in \bbR}$, consisting of \emph{leaves} of the form $\cur{u \in \cX : \pi(u) = x}$ for $x \in \bbR$, each of which is invariant under the flow of \eqref{eq:pde} and terminates at the manifold, see \cite{bates_invariant_2000}.
    However, we will not use this terminology.
\end{remark}
For $v \in \Gamma_{\delta}$, the isochronal phase is now defined to be $\iso(v)$, and the map $\iso \colon \Gamma_{\delta} \to \bbR$ is referred to as the isochron map.
Note that the isochron map $\iso$ should not be confused with the circle constant $\pi$, which has the usual meaning.

Regarding the noisy equation, it should be clear that we cannot expect to have a convergence similar to \eqref{eq:isochronconv} for a solution $u_{\sigma}$ to \eqref{eq:spde}.
However, as long as $u_{\sigma}(t) \in \Gamma_{\delta}$, the isochronal phase $\iso(u_{\sigma})$ is well-defined.
Moreover, from \eqref{eq:isochronacc} it is still sensible to interpret $\iso(u_{\sigma})$ as the position of the perturbed traveling pulse $u_{\sigma}$.

\subsection{Isochron map derivatives}
As was done in \cite{adams_existence_2024,adams_isochronal_2025}, we want to obtain a detailed description of $\iso(u_{\sigma})$ using It\^o's formula.
Hence, it will be necessary to have sufficient regularity of the isochron map and to obtain as much information about the derivatives as possible.
In terms of regularity, the following proposition suffices.
\begin{proposition}[\cite{adams_existence_2024}]
    \label{prop:isochronreg}
    The isochron map $\iso \colon \Gamma_{\delta} \to \bbR$ has five bounded Fr\'echet dervatives in the topology of $\cX$.
\end{proposition}

Since $\iso$ encodes information about the long-term behavior of a nonlinear PDE, it is expected that its derivatives are in general highly nontrivial to compute.
However, we benefit from transferring the symmetries of \eqref{eq:pde} to $\iso$.
From now on, for $x \in \bbR$ we let $\cT_x$ denote the right-translation operator given by $\cT_x f = f(\cdot - x)$.

\begin{lemma}[Symmetry]
    \label{lem:pisym}
    For $x \in \bbR$ and $f \in \Gamma_{\delta}$ we have $\iso(\cT_x f) = x + \iso(f)$.
\end{lemma}
\begin{proof}
    Since $A$ commutes with translations, we have the identity $u^{\cT_x v}(t) = \cT_x [u^v(t)]$. Substituting this into \eqref{eq:isochronconv}, the claim follows.
\end{proof}
As a second key observation, we note that the derivative $\iso'(u^*)$ can be effectively characterized in terms of $\psi$ from \eqref{eq:Picpsi}.
Although similar statements have been shown for different notions of phase, it seems that \cref{lem:dpiexplicit} has not been observed before in the context of the isochronal phase for traveling pulses.
\begin{lemma}[First derivative]
    \label{lem:dpiexplicit}
    For $v \in \cX$ we have the identity
    \begin{equation}
        \label{eq:piprime}
        \iso'(u^*)[v] = \langle \psi,v\rangle,
    \end{equation}
    where $\psi$ is as in \eqref{eq:Picpsi}.
\end{lemma}
\begin{proof}
    By transforming to a co-moving coordinate frame, we may assume without loss of generality that $c = 0$.
    We differentiate $u^v$ (which was defined as the solution to \eqref{eq:pde} with $u^v(0) = v$) with respect to the initial condition.
    We write $u'[w]$ for the derivative of $u^v$ with respect to $v$ in the direction $w$ evaluated at $u^*$.
    By \eqref{eq:pde} and the chain rule, we see that $u'[w]$ satisfies:
    \begin{equation*}
        \diff u'[w] = A u'[w] \rd t + f'(u^*)u'[w]\rd t = \cL u'[w] \rd t,
    \end{equation*}
    with initial condition $u'[w](0) = w$.
    By \cref{ass:pulse} it then follows that $u'[w](t) = P(t)w$ and furthermore:
    \begin{equation}
        \label{eq:dpiexplicitilim1}
        \lim_{t \to \infty} u'[w](t) = \lim_{t \to \infty}\Pi^c w + P(t)\Pi^s w = \Pi^c w = -\langle \psi,w \rangle \partial_x u^*,
    \end{equation}
    where we have used \eqref{eq:Picpsi} for the final identity.
    On the other hand, by \cref{thm:isochron} it holds that $\lim_{t \to \infty} u^v(t) = \cT_{\iso(v)} u^*$.
    Differentiating this in the same way as before we find
    \begin{equation}
        \label{eq:dpiexplicitilim2}
        \lim_{t \to \infty} u'[w](t) = -(\cT_{\iso(u^*)} \partial_x u^*)\iso'(u^*)[w] = -\iso'(u^*)[w] \partial_x u^*.
    \end{equation}
    The claim follows by comparing \eqref{eq:dpiexplicitilim1} and \eqref{eq:dpiexplicitilim2}.
    Although this calculations was formal, it can be made rigorous by using \cref{ass:analytic,ass:pulse} and the mild formulation of \eqref{eq:pde}, since all involved objects have sufficient smoothness.
\end{proof}

The characterization \eqref{eq:piprime} is highly effective from a practical standpoint, since $\psi$ can be computed as the solution to a one-dimensional ODE.
Unfortunately, expressions as nice as \eqref{eq:piprime} are not available for higher derivatives of $\iso$.
Although one can extend the proof method of \cref{lem:dpiexplicit} to second order to compute $\iso''(u^*)$, this results in an expression which involves a convolution with $P(t)$ similar to \cite[(2.46)]{hamster_stability_2020}.
Note that calculating the second derivative is not a mere curiosity, since it has been observed that $\iso''(u^*)$ determines (to leading order) the noise-induced change in speed of the traveling pulse \cite[Theorem 2.3]{giacomin_small_2018} \cite[\S 2.2]{hamster_stability_2020} \cite[\S 4.2]{adams_isochronal_2025}.

However, the following symmetry properties (which are essentially inherited from the translational symmetry of \eqref{eq:pde}) suffice for our proofs.
\begin{lemma}[Symmetry of derivatives]
    \label{lem:piderivsym}
    For $x \in \bbR$, $u \in \Gamma_{\delta}$ and $v,w\in \cX$ we have
    \begin{subequations}
    \label{eq:piderivsym}
    \begin{align}
        \label{eq:dpisym}
        \iso'(\cT_x u)[v] &= \iso'(u)[\cT_{-x}v], \\
        \label{eq:ddpisym}
        \iso''(\cT_x u)[v,w] &= \iso''(u)[\cT_{-x}v,\cT_{-x}w],
    \end{align}
    \end{subequations}
    and similar for higher derivatives.
\end{lemma}
\begin{proof}
    Differentiate the identity from \cref{lem:pisym}.
\end{proof}

Combining \eqref{eq:piprime} and \eqref{eq:dpisym} we also find the identity
\begin{equation}
    \label{eq:dpiustartrans}
    \iso'(\cT_x u^*)[v] = \langle \psi,\cT_{-x}v\rangle = \langle \cT_x \psi, v \rangle ,\qquad x \in \bbR,\,v \in \cX,
\end{equation}
which will be used in the sequel.

\subsection{Reduced phase SDE}
\label{subsec:dothereduction}
With the isochron derivatives characterized, we begin the analysis of the dynamics of the isochronal phase of the solution $u_{\sigma}$ to \eqref{eq:spde}.
The aim is to derive an approximate SDE for $\iso(u_{\sigma})$ which is autonomous and does not involve the `full' solution $u_{\sigma}$.

Converting \eqref{eq:spde} to the equivalent It\^o formulation, we see that $u_{\sigma}$ satisfies:
\begin{equation*}
    \diff u_{\sigma} = A u_{\sigma} \rd t + f(u_{\sigma}) \rd t + \tfrac{1}{2}\sigma^2 \sum_{k \in \bbZ} \alpha_k^2\,  g'(u_{\sigma})g(u_{\sigma})e_k^2 \rd t + \sigma \sum_{k \in \bbZ}\alpha_k g(u_{\sigma})e_k \rd \beta_k(t),
\end{equation*}
where the stochastic integral is interpreted in the (mild) It\^o sense.
Motivated by the coming application of It\^o's formula, we now define $\mathfrak{a} \colon \Gamma_{\delta} \to \bbR$ and $\mathfrak{b}_k \colon \Gamma_{\delta} \to \bbR$ by
\begin{subequations}
\label{eq:deffrakab1}
\begin{align}
    \label{eq:deffraka1}
    \mathfrak{a}(u) &\coloneq \tfrac{1}{2}\sum_{k \in \bbZ} \alpha_k^2 \bra[\big]{\iso'(u)[g'(u)g(u) e_k^2] + \iso''(u)[g(u)e_k,g(u)e_k]}, \\
    \label{eq:deffrakb1}
    \mathfrak{b}_k(u) &\coloneq \alpha_k \iso'(u)[g(u)e_k], \qquad k \in \bbZ.
\end{align}
\end{subequations}

Applying the It\^o formula shown by \textcite{adams_isochronal_2025} to $\pi(u_{\sigma})$, we then see that the isochronal phase satisfies
\begin{equation}
    \label{eq:piuito}
     \diff\iso(u_{\sigma}) = c \rd t 
     + \sigma^2\mathfrak{a}(u_{\sigma}) \rd t 
     + \sigma\sum_{k \in \bbZ}\mathfrak{b}_k(u_{\sigma}) \rd \beta_k(t).
\end{equation}
At this point the right-hand side of \eqref{eq:piuito} still involves the full solution $u_{\sigma}$.
However, motivated by \cref{thm:isochron}-\eqref{eq:isochronacc} and the expectation that $u_{\sigma} \in \Gamma_{\delta}$ with high probability, we now \emph{postulate} the approximation $u_{\sigma} \approx \cT_{\iso(u_{\sigma})}u^*$.
Substituting this into \eqref{eq:piuito} and letting $\gamma_{\sigma}$ denote the resulting approximation to $\iso(u_{\sigma})$, we obtain the following SDE for $\gamma_{\sigma}$:
\begin{equation}
    \label{eq:sdegamma}
     \diff\gamma_{\sigma} = c \rd t 
     + \sigma^2\mathfrak{a}(\cT_{\gamma_{\sigma}} u^*) \rd t 
     + \sigma\sum_{k \in \bbZ}\mathfrak{b}_k(\cT_{\gamma_{\sigma}} u^*) \rd \beta_k(t).
\end{equation}
With \eqref{eq:sdegamma} we have achieved our goal of deriving an autonomous SDE for an approximation to $\iso(u_{\sigma})$.
Thus, if we can prove that the approximation $\gamma_{\sigma} \approx \iso(u_{\sigma})$ is accurate (to a degree which will be specified shortly), we have successfully reduced the dynamics of $u_{\sigma}$.
For this, the following regularity properties of the coefficients of \eqref{eq:piuito}-\eqref{eq:sdegamma} are needed.
\begin{proposition}
    \label{prop:frakabreg1}
    We have $\mathfrak{a} \in C^3(\Gamma_{\delta})$ and $\mathfrak{b} \in \ell^2(\bbZ;C^4(\Gamma_{\delta}))$.
\end{proposition}
\begin{proof}
    Recalling the definitions of $e_k$ \eqref{eq:basis} and $\nu$ (\cref{ass:noise}), 
    we first show that we have the estimate
    \begin{equation}
        \label{eq:mult}
        \norm{f e_k}_{\cX} \lesssim \abs{k}^{\nu} \norm{f}_{\cX}, \qquad f \in \cX,\, k \in \bbZ.
    \end{equation}
    If $k=1$ and $\cX$ is replaced by $W^{n,p}$ with $n \in \bbN, p \in [1,\infty]$, then \eqref{eq:mult} holds by basic calculus.
    By complex interpolation, \eqref{eq:mult} then also holds with $k=1$ and with $\cX$ as in \cref{ass:analytic}.
    The case $k \in \bbZ$ follows by rescaling space, since $\norm{f(\lambda\,\cdot)}_\cX \lesssim (1+\lambda^{\nu}) \norm{f}_\cX$.
    
    Using \eqref{eq:mult}, it follows from \cref{prop:isochronreg} and the chain rule that we have
    \begin{align*}
        \norm{u \mapsto \iso'(u)[g'(u)g(u)e_k^2] + \iso''(u)[g(u)e_k,g(u)e_k]}_{C^3(\Gamma_{\delta})} &\lesssim \abs{k}^{2\nu}, \\
        \norm{u \mapsto \iso'(u)[g(u)e_k]}_{C^4(\Gamma_{\delta})} &\lesssim \abs{k}^{\nu},
    \end{align*}
    for every $k \in \bbZ$.
    The regularity of $\mathfrak{a}$ and $\mathfrak{b}$ then follows from \eqref{eq:noisereg} and \eqref{eq:deffrakab1}.
\end{proof}
We now refer ahead to \cref{prop:frakabreg2} to see that \cref{prop:frakabreg1} also guarantees that the coefficients of \eqref{eq:sdegamma} are smooth.
Hence, for $x \in \bbR$ we will write $\gamma^x_{\sigma}$ for the unique solution to \eqref{eq:sdegamma} with initial condition $\gamma^x_{\sigma}(0) = x$.

\subsection{Validity of the approximation}
We will now confirm the validity of the approximation $u_{\sigma} \approx \cT_{\gamma_{\sigma}}u^*$, where $\gamma_{\sigma}$ is defined via \eqref{eq:sdegamma}.
Since our ultimate goal is to transfer the synchronization behavior of $\gamma_{\sigma}$ back to $u_{\sigma}$, the approximation needs to be valid for longer than the time it takes $\gamma_{\sigma}$ to synchronize.
As the characteristic time scale of \eqref{eq:sdegamma} is $t \sim \sigma^{-2}$ we see that if there is any hope to succeed, the approximation $\gamma_{\sigma} \approx \iso(u_{\sigma})$ must hold \emph{at least} until $t \sim \sigma^{-2}$ and preferably longer.
We now show that this is indeed the case, and the validity holds on the (slightly) longer time scale $t \sim \sigma^{-2}\log(\sigma^{-1})$.

Before we proceed to the mathematical statement, let us remark on the initial conditions of \eqref{eq:spde} and \eqref{eq:sdegamma}.
Naturally, the initial condition for $\gamma_{\sigma}$ needs to be compatible with the initial condition for $u_{\sigma}$.
Recalling that $u^x_{\sigma}(0) = \cT_x u^*$ by definition and $\iso(\cT_x u^*) = x$ by \cref{thm:isochron}, we should of course use $\gamma^x_{\sigma}$ to approximate $\iso(u^x_{\sigma})$.

\begin{theorem}
    \label{thm:longtermstab}
    Let $q \in (0,2/3)$ and $x \in \bbR$.
    Then we have
    \begin{equation}
        \label{eq:longtermstab}
        \PP[\Big]{\sup_{t \in [0,\exp(\sigma^{-2+3q})]} \norm{u^x_{\sigma}(t) - \cT_{\iso(u^x_{\sigma}(t))}u^*}_\cX \geq \sigma^{q}} \leq \exp(-\sigma^{-2+3q})
    \end{equation}
    for all $\sigma \ll_q 1$.
\end{theorem}

\begin{theorem}
    \label{thm:longtermapprox}
    Let $q \in (0,2/9)$ and $x \in \bbR$. Then we have
    \begin{equation}
        \label{eq:longtermapprox}
        \PP[\Big]{\sup_{t \in [0,\sigma^{-2}\log(\sigma^{-1})^{1-q}]} \abs{\iso(u^x_{\sigma}(t)) - \gamma^x_{\sigma}(t)} \geq \sigma^q} \leq \sigma^{q}
    \end{equation}
    for all $\sigma \ll_q 1$.
\end{theorem}
Recall that the notation $\sigma \ll_q 1$ means that the relevant estimates hold whenever $\sigma \leq c$ for some (small) constant $c$ which depends on $q$ and the objects from \cref{ass:analytic,ass:pulse,ass:noise}.
\begin{remark}
    As $q$ increases, the estimates \eqref{eq:longtermstab}-\eqref{eq:longtermapprox} become increasingly suboptimal.
    However, since our aim with \cref{thm:longtermstab,thm:longtermapprox} is to obtain approximations which are valid on the \emph{longest} possible time scale, we intend for $q$ to be small.
    More flexible versions of the estimates can be formulated, but this complicates the presentation and does not lead to any improvement of \cref{thm:spdesync}.
\end{remark}

\begin{remark}
    Comparing \eqref{eq:longtermstab} with \eqref{eq:longtermapprox}, we see that the approximation $u_{\sigma}(t) \approx \cT_{\iso(u_{\sigma}(t))}u^*$ holds on a much longer time scale than the approximation $\iso(u_{\sigma}(t)) \approx \gamma_{\sigma}(t)$.
    Thus, the reduced SDE \eqref{eq:sdegamma} is only accurate for a fraction of the typical lifetime of the pulse.
\end{remark}

\begin{proof}[Proof of \cref{thm:longtermstab}]
    Let $C$ be the constant from \eqref{eq:isochronacc}.
    For $\sigma > 0$, choose $\eps = C^{-1}\sigma^{q}$ and $T = \exp(\sigma^{-2+3q})$ in \cite[Theorem 4.9]{vanwinden_noncommutative_2024}.
    After scaling away the constants from the theorem against appropriate powers of $\sigma^q$, we see that
    \begin{equation*}
        \PP[\Big]{u^x_{\sigma}(t) \in \Gamma_{C^{-1}\sigma^q} \text{ for all } t \in [0,\exp(\sigma^{-2+3q})]} \geq 1-\exp(-\sigma^{-2+3q}),
    \end{equation*}
    for $\sigma \ll_q 1$.
    The desired estimate \eqref{eq:longtermstab} then follows using \eqref{eq:isochronacc}.
\end{proof}

\begin{proof}[Proof of \cref{thm:longtermapprox}]
    To avoid cluttering the notation we write $u$ and $\gamma$ instead of $u^x_{\sigma}$ and $\gamma^x_{\sigma}$ throughout the proof.
    For $\sigma > 0$ we define the stopping time
    \begin{equation}
        \label{eq:approxtausigma}
        \tau_{\sigma} \coloneq \sup \cur[\big]{t \in [0,T] : \norm{u(t) -\cT_{\iso(u(t))}u^*}_\cX \leq \sigma^{3q}}.
    \end{equation}
    When $\sigma \ll_q 1$ it then holds that $u(t) \in \Gamma_{\delta}$ for all $t \in [0,\tau_{\sigma}]$ ($\delta$ is as in \cref{thm:isochron}).
    Thus, from \cref{prop:frakabreg1} we see that for $t \in [0,\tau_{\sigma}]$ we have
    \begin{align}
        &\abs{\mathfrak{a}(u(t)) - \mathfrak{a}(\cT_{\gamma(t)}u^*)} 
        + \norm{\mathfrak{b}(u(t)) - \mathfrak{b}(\cT_{\gamma(t)}u^*)}_{\ell^2(\bbZ)} \nonumber\\
        &\qquad \lesssim \norm{u(t) - \cT_{\gamma(t)} u^*}_{\cX} \nonumber\\
        &\qquad \leq \norm{u(t) - \cT_{\iso(u(t))}u^*}_{\cX} + \norm{\cT_{\iso(u(t))}u^* - \cT_{\gamma(t)} u^*}_{\cX} \nonumber\\
        \label{eq:frakablip}
        &\qquad \lesssim \sigma^{3q} + \abs{\iso(u(t)) - \gamma(t)},
    \end{align}
    where we have used \eqref{eq:isochronacc} and \cref{rem:translateustar} for the final step.
    Writing $X(t) = \iso(u(t)) - \gamma(t)$, we see from \eqref{eq:piuito}, \eqref{eq:sdegamma}, and It\^o's formula that 
    \begin{align*}
        \diff \abs{X(t)}^2 &= 2\sigma^2 X(t)(\mathfrak{a}(u(t)) - \mathfrak{a}(\cT_{\gamma(t)}u^*))\rd t  \\
        &\quad+ \sigma^2 \norm{\mathfrak{b}(u(t)) - \mathfrak{b}(\cT_{\gamma(t)}u^*)}_{\ell^2(\bbZ)}^2 \rd t \\
        &\quad+ 2\sigma X(t)\sum_{k \in \bbZ}(\mathfrak{b}_k(u(t)) - \mathfrak{b}_k(\cT_{\gamma(t)}u^*))\rd \beta_k(t).
    \end{align*}
    Writing $X^*(t) = \sup_{s \in [0,t]} \ind_{[0,\tau_{\sigma}]}(s) \abs{X(s)}$, we thus obtain from \eqref{eq:frakablip} and the Burkholder--Davis--Gundy inequality (note that $X^*(0) = 0$ by our choice of initial condition):
    \begin{align*}
        \EE{X^*(t)^2} 
        &\lesssim \EE[\Big]{\sigma^2 \int_0^{\tau_{\sigma}} \abs{X(s)}(\sigma^{3q} + \abs{X(s)}) + (\sigma^{3q} + \abs{X(s)})^2 \rd s} \\
        &\quad+ \EE[\Big]{\bra[\Big]{\int_0^{\tau_{\sigma}} \sigma^2 \abs{X(s)}^2(\sigma^{3q} + \abs{X(s)})^2\rd s}^{1/2}} \\
        &\lesssim t\sigma^{2+6q} + \sigma^2\int_0^t \EE{ X^*(s)^2} \rd s
        +\EE[\big]{X^*(t) \bra[\big]{\sigma^2\int_0^t \sigma^{6q} + X^*(s)^2 \rd s}^{1/2}}
        \\
        &\lesssim (1+\eps^{-1})\bra[\big]{t\sigma^{2+6q} + \sigma^2 \int_0^t \EE{X^*(s)^2}\rd s} + \eps\EE{X^*(t)^2},
    \end{align*}
    where we have used the inequality $2 xy \leq \eps x^2 + \eps^{-1}y^2$ for the final step.
    Choosing $\eps$ sufficiently small (independent of $\sigma$) allows us to absorb the rightmost term into the left-hand side, after which we apply Gr\"onwall's inequality  to find
    \begin{equation*}
        \EE{X^*(t)^2 } \lesssim t \sigma^{2+6q}e^{2C\sigma^2 t}, \qquad t \geq 0,
    \end{equation*}
    for some absolute constant $C>0$.
    It then follows from Markov's inequality and the definition of $X^*$ that
    \begin{equation*}
        \PP[\Big]{\sup_{t \in [0,T \wedge \tau_{\sigma}]} \abs{\iso(u(t)) - \gamma(t)} \geq \sqrt{T}\sigma^{1+2q}e^{C\sigma^2 T}} \lesssim \sigma^{2q} \leq \tfrac{1}{2}\sigma^{q}
    \end{equation*}
    for any $T > 0$ and $\sigma \ll_q 1$.
    With the choice $T = \sigma^{-2}\log(\sigma^{-1})^{1-q}$ we also have
    \begin{equation*}
        \sqrt{T}\sigma^{1+2q} e^{C\sigma^2 T} = \sigma^{2q}\log(\sigma^{-1})^{\tfrac{1-q}{2}}e^{C \log(\sigma^{-1})^{1-q}}
        \leq \sigma^q
    \end{equation*}
    for $\sigma \ll_q 1$, so that
    \begin{equation*}
        \PP[\Big]{\sup_{t \in [0,\sigma^{-2}\log(\sigma^{-1})^{1-q} \wedge \tau_{\sigma}]} \abs{\iso(u(t)) - \gamma(t)} \geq \sigma^q} \leq \tfrac{1}{2}\sigma^{q}.
    \end{equation*}
    Recalling \eqref{eq:approxtausigma} we can now combine with \eqref{eq:longtermstab} (note that $3q < 2/3$ by assumption) and a union bound to lift the restriction $t \leq \tau_{\sigma}$, and the result follows.
\end{proof}

\section{The reduced SDE}
\label{sec:reducedsdeanalysis}
With the validity of the approximation $u_{\sigma} \approx \cT_{\gamma_{\sigma}} u^*$ firmly established, we turn our attention towards analyzing the SDE which defines $\gamma_{\sigma}$ \eqref{eq:sdegamma}.
The goal of this section is to establish (weak) synchronization of $\gamma_{\sigma}$.
In order to transfer the synchronization to $u_{\sigma}$, the synchronization moreover needs to happen on the time scale $t \sim \sigma^{-2}$, and the rate needs to be uniform in the initial condition.

Although synchronization of one-dimensional SDEs has been extensively studied (see \cref{subsec:introsync}), we were unable to find a statement in the literature which exactly fits our setting.
Hence, our strategy is to verify the abstract criteria set forth in the work by \textcite{flandoli_synchronization_2017}.
This involves checking a mixing condition (\cref{subsec:sdeergodic}), an asymptotic stability condition (\cref{subsec:sdestable}), and some irreducibility conditions (\cref{subsec:sdecontrol}).
The most critical part of the argument takes place in \cref{subsec:sdesync}:
we use \cref{ass:scaling} to quantify the synchronization rate in terms of $\sigma$, and use uniform mixing properties of \eqref{eq:sdegamma} to show that the synchronization rate is uniform in the initial condition.

\subsection{Properties of the coefficients}
We start our analysis by suggestively defining $\mathfrak{a} \colon \bbR \to \bbR$ and $\mathfrak{b} \colon \bbR \times \bbZ \to \bbR$ via
\begin{subequations}
\label{eq:deffrakab2}
\begin{align}
    \label{eq:deffraka2}
    \mathfrak{a}(x) &\coloneq \mathfrak{a}(\cT_x u^*), & & x \in \bbR, \\
    \label{eq:deffrakb2}
    \mathfrak{b}_k(x) &\coloneq \mathfrak{b}_k(\cT_x u^*), & & x \in \bbR,\,k \in \bbZ,
\end{align}
\end{subequations}
where the right-hand side is interpreted according to \eqref{eq:deffrakab1}.
Despite our abuse of notation there should not be any confusion between \eqref{eq:deffrakab1} and \eqref{eq:deffrakab2}, as the latter definition will exclusively be used in the rest of this section.
With the new definition, we may write \eqref{eq:sdegamma} in the It\^o formulation as
\begin{equation}
    \label{eq:sdegammasimp}
    \diff \gamma_{\sigma} = c \rd t 
    + \sigma^2 \mathfrak{a}(\gamma_{\sigma}) \rd t 
    + \sigma \sum_{k \in \bbZ}\mathfrak{b}_k(\gamma_{\sigma}) \rd \beta_k(t),
\end{equation}
or in the equivalent Stratonovich formulation as
\begin{equation}
    \label{eq:sdegammasimpstrat}
     \diff \gamma_{\sigma} = c \rd t 
    + \sigma^2 \bra[\big]{\mathfrak{a}(\gamma_{\sigma}) - \tfrac{1}{2}\sum_{k \in \bbZ}\mathfrak{b}_k'(\gamma_{\sigma})\mathfrak{b}_k(\gamma_{\sigma})} \rd t 
    + \sigma \sum_{k \in \bbZ}\mathfrak{b}_k(\gamma_{\sigma}) \circ \diff \beta_k(t).   
\end{equation}
\begin{remark}
    \label{rem:wrongreduction}
    One might expect that the extra term in \eqref{eq:sdegammasimpstrat} will cancel with the first term of $\mathfrak{a}$ in \eqref{eq:deffraka1}, since both terms originate from an It\^o--Stratonovich correction.
    However, this is generally not the case.
    For example, when $W(t,x) = \beta_0(t)$ the former term vanishes (since $\mathfrak{b}_0' \equiv 0$ by \eqref{eq:frakbexplicit}) but the latter does not.
    This demonstrates the subtle point that naively performing a phase reduction in the Stratonovich formulation leads to an inaccurate approximation.
    In the physics literature, this has been observed in the context of phase reduction for nonlinear oscillators \cite{yoshimura_phase_2008,teramae_stochastic_2009}.
\end{remark}
\begin{proposition}
    \label{prop:frakabreg2}
    We have $\mathfrak{a} \in C^3(\bbR)$ and $\mathfrak{b} \in \ell^2(\bbZ;C^4(\bbR))$.
\end{proposition}
\begin{proof}
    This follows from the definition \eqref{eq:deffrakab2} using \cref{prop:frakabreg1} and the chain rule, taking into account that $x \mapsto \cT_x u^*$ is four times differentiable (with values in $\cX$) by \cref{ass:pulse}.
\end{proof}

We now gather some more facts about the coefficients aside from the smoothness.
We will not concern ourselves too much with the exact form of $\mathfrak{a}$, since it turns out to have no discernible effect on the synchronization properties of \eqref{eq:sdegammasimp}.
Instead, the synchronization is facilitated mainly through the multiplicative noise coefficients $\mathfrak{b}_k$, which can be made much more explicit using \cref{lem:dpiexplicit,lem:piderivsym}.
Indeed, using \eqref{eq:dpiustartrans} we find from \eqref{eq:deffrakb1} and \eqref{eq:deffrakb2} that
\begin{subequations}
\begin{equation}
    \label{eq:frakbexplicit}
    \mathfrak{b}_k(x) = \alpha_k \langle \psi g(u^*),\cT_{-x}e_k\rangle = \alpha_k \langle \cT_x[\psi g(u^*)], e_k \rangle, \qquad x \in \bbR,\, k \in \bbZ,
\end{equation}
where we recall that $\psi$ is as in \eqref{eq:Picpsi}.
Recalling the definition of $e_k$ \eqref{eq:basis} and using trigonometric addition formulas, we also find for $k \in \bbN$:
\begin{equation}
    \label{eq:frakbmatrix}
    \begin{pmatrix}
        \mathfrak{b}_k(x) \\
        \mathfrak{b}_{-k}(x)
    \end{pmatrix}
    =
    \begin{pmatrix}
        \alpha_k & 0 \\ 0 & \alpha_{-k}
    \end{pmatrix}
    \begin{pmatrix}
        \cos(2\pi kx) & -\sin(2\pi kx) \\
        \sin(2\pi kx) & \cos(2\pi kx)
    \end{pmatrix}
    \begin{pmatrix}
        c_k \\ c_{-k}
    \end{pmatrix},
\end{equation}
where we have written $c_k \coloneq \langle \psi g(u^*),e_k\rangle$.
This then leads to
\begin{equation}
    \label{eq:simplenoise}
    \sum_{k \in \bbZ} \mathfrak{b}_k(x) \rd \beta_k
    = 
    \alpha_0 \langle \psi g(u^*), e_0 \rangle \rd \beta_0 + 
    \sum_{k \in \bbN} 
    \begin{pmatrix}
        c_k \\
        c_{-k}
    \end{pmatrix}^{\top}
    R(-2\pi kx) 
    \begin{pmatrix}
        \alpha_k & 0 \\ 0 & \alpha_{-k}
    \end{pmatrix}
    \begin{pmatrix}
        \rd \beta_k \\ \rd \beta_{-k}
    \end{pmatrix},
\end{equation}
\end{subequations}
where $R(\theta)$ denotes the usual rotation matrix which rotates $\bbR^2$ counterclockwise by $\theta$ radians.
Notice that we have fully isolated the dependence on $x$ into the rotation matrices.
The nondegeneracy condition in \cref{ass:noise} then affords us the following lemma.
\begin{lemma}
    \label{lem:sinspan}
    We have $\mathfrak{b}_1^2(x) + \mathfrak{b}_{-1}^2(x) > 0$ for every $x \in \bbR$, and also
    \begin{equation}
        \label{eq:sinspan}
        \Span \cur{\mathfrak{b}_{\pm 1}(\cdot)} = \Span \cur{\cos(2\pi \cdot),\sin(2\pi \cdot)} = \cur{L \sin(2\pi\cdot + \eta) : L,\eta \in \bbR }.
    \end{equation}
\end{lemma}
\begin{proof}
    The second identity in \eqref{eq:sinspan} is well-known, and from this and \eqref{eq:noisenondegen} it follows that $c_1^2 + c_{-1}^2 > 0$.
    Since $\alpha_{\pm1 }$ are both nonzero by \cref{ass:noise}, the remaining claims can be read off from \eqref{eq:frakbmatrix}.
\end{proof}

We also observe the following symmetry properties of $\mathfrak{a}$ and $\mathfrak{b}$, which are inherited from symmetries of the PDE and the noise.
\begin{proposition}
    \label{prop:frakabsym}
    For every $x \in \bbR$ and $k \in \bbZ$, we have
    \begin{subequations}
    \label{eq:frakabsym}
    \begin{equation}
        \label{eq:frakabper}
        \mathfrak{a}(x) = \mathfrak{a}(x+1), \qquad \mathfrak{b}_k(x) = \mathfrak{b}_k(x+1).
    \end{equation}
    If additionally $\alpha_k = \alpha_{-k}$ for every $k \in \bbZ$, then we also have
    \begin{equation}
        \label{eq:frakabtiv}
        \mathfrak{a}(x) = \mathfrak{a}(0), \qquad
        \norm{\mathfrak{b}(x)}_{\ell^2(\bbZ)} = \norm{\mathfrak{b}(0)}_{\ell^2(\bbZ)}.
    \end{equation}
    \end{subequations}
\end{proposition}
\begin{proof}
    The identities involving $\mathfrak{b}$ can be read off directly from \eqref{eq:frakbexplicit}-\eqref{eq:frakbmatrix}.
    Regarding $\mathfrak{a}$, note that by \eqref{eq:deffraka1}-\eqref{eq:deffraka2} and \cref{lem:piderivsym} we have
    \begin{equation*}
        \mathfrak{a}(x) = \tfrac{1}{2}\sum_{k \in \bbZ}\alpha_k^2\bra[\big]{ \iso'(u^*)[g'(u^*)g(u^*) \cT_{-x} e_k^2]
        + \iso''(u^*)[g(u^*)\cT_{-x}e_k,g(u^*)\cT_{-x} e_k]}.
    \end{equation*}
    The periodicity of $\mathfrak{a}$ immediately follows, so it only remains to show the first identity of \eqref{eq:frakabtiv} in the case where $\alpha_k = \alpha_{-k}$ for all $k \in \bbN$.
    For this, it suffices to note that
    \begin{align*}
        &\sum_{k = \pm n}\alpha_k^2\bra[\big]{ \iso'(u^*)[g'(u^*)g(u^*) \cT_{-x} e_k^2]
        + \iso''(u^*)[g(u^*)\cT_{-x}e_k,g(u^*)\cT_{-x} e_k]} \\
        &\quad = \alpha_n^2 \sum_{k = \pm n}\bra[\big]{ \iso'(u^*)[g'(u^*)g(u^*) \cT_{-x} e_k^2]
        + \iso''(u^*)[g(u^*)\cT_{-x}e_k,g(u^*)\cT_{-x} e_k]} \\
        &\quad = \alpha_n^2 \bra[\big]{2\iso'(u^*)[g'(u^*)g(u^*)]
        + \sum_{k = \pm n}\iso''(u^*)[g(u^*)e_k,g(u^*)e_k]},
    \end{align*}
    for every $n \in \bbN$ and $x \in \bbR$, where the final identity follows (for the first term) since $\sin(x)^2 + \cos(x)^2 = 1$ and (for the second term) by trigonometric addition formulas, noting that $\iso''(u^*)$ is symmetric.
\end{proof}

\subsection{RDS generation and ergodicity}
\label{subsec:sdeergodic}
By the periodicity of $\mathfrak{a}$ and $\mathfrak{b}$ \eqref{eq:frakabper}, we can now interpret \eqref{eq:sdegammasimp} as an SDE on either $\bbT$ or $\bbR$.
Combined with smoothness of the coefficients, this implies that \eqref{eq:sdegammasimp} generates a $C^2$ white noise RDS on $\bbT$ as well as on $\bbR$.
Recall that the maps $(\theta_t)_{t \in \bbR}$ apply a time shift to $(\beta_k)_{k \in \bbZ}$ and were defined in \cref{subsec:probrds}.
\begin{proposition}[RDS generation]
    \label{prop:rds}
    Let $S$ be either $\bbR$ or $\bbT$.
    There exists a perfect cocycle $\phi_{\sigma} \colon \bbR^+ \times \Omega \times S \to S$ over $(\theta_t)_{t \in \bbR}$ which satisfies the following properties:
    \begin{enumerate}[label=(\roman*)]
        \item \label{it:rds:white}$\omega \mapsto \phi_{\sigma}(t,\theta_s \omega,x)$ is $\cF_{s,s+t}$-measurable for every $x \in S$, $s \in \bbR$, $t \in \bbR^+$.
        \item $t \mapsto \phi_{\sigma}(t,\omega,x) \eqcolon \gamma^x_{\sigma}(t)$ solves \eqref{eq:sdegammasimp} with $\gamma_{\sigma}^x(0) = x$, for every $x \in S$.
        \item $(t,x) \mapsto \phi(t,\omega,x)$ is jointly continuous for every $\omega \in \Omega$.
        \item \label{it:rds:c2} $x \mapsto \phi_{\sigma}(t,\omega,x) \in C^2(S;S)$ for every $t \in \bbR^+$, $\omega \in \Omega$.
        \item \label{it:rds:integrable} $\EE{\norm{x \mapsto \phi_{\sigma}(1,\omega,x)}_{C^2(S;S)}} < \infty$.
    \end{enumerate}
\end{proposition}
\begin{remark}
    We will not always explicitly specify whether we consider $\phi_{\sigma}$ as an RDS on $\bbT$ or $\bbR$.
    However, the notation $d(x,y)$ will \emph{always} refer to the distance on the torus.
    In the case where $x,y \in \bbR$ a priori, one should interpret $d(x,y)$ as the distance between the equivalence classes (modulo $2\pi$) of $x$ and $y$.
\end{remark}
\begin{proof}
    \ref{it:rds:white}-\ref{it:rds:c2} follow from \cref{prop:frakabreg2} and \cite[Theorem 28]{arnold_perfect_1995}, and additionally \eqref{eq:frakabper} in the case $S = \bbT$.
    For \ref{it:rds:integrable} it suffices by periodicity to prove the case $S = \bbT$.
    We differentiate \eqref{eq:sdegammasimp} using the chain rule to find that $\partial_x \gamma^x_{\sigma}$ satisfies the following SDE:
    \begin{equation}
        \label{eq:sdepartialxgammax}
        \diff (\partial_{x}\gamma^x_{\sigma}) 
        = \sigma^2\mathfrak{a}'(\gamma^x_{\sigma})\partial_x \gamma^x_{\sigma} \rd t
        + \sigma \sum_{k \in \bbZ} \mathfrak{b}_k'(\gamma^x_{\sigma})\partial_x \gamma^x_{\sigma}  \rd \beta_k(t)
    \end{equation}
    with initial condition $\partial_x \gamma^x_{\sigma}(0) = 1$, with similar equations being satisfied by $\partial_x^{(2)} \gamma^x_{\sigma}$ and $\partial_x^{(3)} \gamma^x_{\sigma}$.
    Applying well-known theory for SDEs with Lipschitz coefficients, it then follows that $\sup_{x \in \bbT} \EE{\abs{\partial_x^{(n)}\gamma^x_{\sigma}(1)}} < \infty$ for $n \in \cur{0,1,2,3}$.
    Hence, by the Sobolev embedding and Fubini's theorem we have
    \begin{equation*}
        \EE{\norm{x \mapsto \gamma^x_{\sigma}(1)}_{C^2(\bbT;\bbT)}} 
        \lesssim \EE{\norm{x \mapsto \gamma^x_{\sigma}(1)}_{W^{3,1}(\bbT;\bbT)}}
        = \int_{\bbT}\EE[\Big]{\sum_{k=0}^{3} \abs{\partial_x^{(k)} \gamma^x_{\sigma}(1)}} \rd x < \infty. \qedhere
    \end{equation*}
\end{proof}
    We now let $(P_t f)(x) \coloneq \EE[]{f(\phi_{\sigma}(t,\omega,x))}$ be the Markov semigroup on $C_b(\bbT)$ associated with $\phi_{\sigma}$, and formulate the following ergodicity properties:
\begin{proposition}[Ergodicity]
    \label{prop:mixing}
    $P_t$ has a unique invariant measure $\mu = p \rd x$, where $p \in C^2(\bbT)$ is strictly positive.
    Moreover, $\mu$ is exponentially mixing in the sense that there exist $C,a > 0$ such that
    \begin{equation}
        \label{eq:uniformmixing}
        \norm{P_t f - \int_{\bbT} f \rd \mu}_{C_b(\bbT)} \leq Ce^{-at}\norm{f}_{C_b(\bbT)}, \qquad t \geq 0,
    \end{equation}   
    for every $f \in C_b(\bbT)$.
\end{proposition}

\begin{remark}
    \label{rem:TV}
    The condition \eqref{eq:uniformmixing} is quite strong and might not be satisfied in different settings.
    However, \eqref{eq:uniformmixing} is only used once in the proof of \cref{thm:phisigmasync}, where the weaker condition $\lim_{t \to \infty}\norm{P_t f(\cdot) - \int_{\bbT} f\rd \mu}_{C_b(\bbT)} = 0$ (with convergence rate depending on $f \in C_b(\bbT)$) would already suffice.
\end{remark}
\begin{proof}
    The generator of $P_t$ is given by
    \begin{subequations}
    \begin{equation}
        \cL f = \tfrac{1}{2}\sigma^2\sum_{k \in \bbZ} \mathfrak{b}_k^2\partial_{xx}f + (c+\sigma^2\mathfrak{a}) \partial_x f,
    \end{equation}
    with formal adjoint
    \begin{equation}
        \label{eq:FPgenerator}
        \cL^* p = \tfrac{1}{2}\sigma^2\sum_{k \in \bbZ} \partial_{xx}(\mathfrak{b}_k^2 p) - \partial_x ((c+\sigma^2\mathfrak{a}) p).
    \end{equation}
    \end{subequations}
    We also have $\mathfrak{a} \in C^3(\bbT)$ and $x \mapsto \norm{\mathfrak{b}(x)}_{\ell^2(\bbZ)} \in C^3(\bbT)$ by \cref{prop:frakabreg2}, as well as $\inf_{x \in \bbT} \sum_{k \in \bbZ} \mathfrak{b}_k(x)^2 > 0$ by \cref{lem:sinspan}.
    It follows from well-known Schauder theory (see e.g. \cite{jost_partial_2013}) that $\cL^* p = 0$ has a solution $p \in C^2(\bbT)$ with $p > 0$, which means $\mu \coloneq p \rd x$ is an invariant measure.
    The exponential mixing \eqref{eq:uniformmixing}, which also implies uniqueness of $\mu$, can be seen from e.g.\ \cite{bellet_ergodic_2006} or \cite{mattingly_ergodicity_2002}.
\end{proof}

\subsection{Asymptotic stability}
\label{subsec:sdestable}
To prove asymptotic stability, we will show (strict) negativity of the \emph{Lyapunov exponent}, defined as
\begin{equation}
    \lambda \coloneq \lim_{t \to \infty}  t^{-1}\log \abs{\partial_{x}\phi_{\sigma}(t,\omega,x)}.
\end{equation}
Although $\lambda$ might a priori depend on $\omega$ and $x$, Oseledets' multiplicative ergodic theorem ensures that $\lambda$ is constant $\bbP \times \mu$-almost everywhere.
Moreover, expressions for $\lambda$ in terms of the invariant measure $\mu$ and the coefficients $\mathfrak{a}$ and $\mathfrak{b}$ are known.
In order to make our presentation more self-contained, we now prove these identities `by hand', relying on the Birkhoff pointwise ergodic theorem.
The reader who is familiar with these identities may skip the following proof.
We remark that the second identity in \eqref{eq:sdelyapunov} was seemingly only recently discovered by \textcite{bedrossian_regularity_2022}.

\begin{proposition}[Lyapunov exponent]
    The identity
    \begin{equation}
        \label{eq:sdelyapunov}
        \lambda = 
        \sigma^2 \int_{\bbT}\mathfrak{a}' -\tfrac{1}{2}\sum_{k \in \bbZ}(\mathfrak{b}_k')^2 \rd \mu = 
        - \tfrac{1}{2}\sigma^2 \sum_{k \in \bbZ} \int_{\bbT}\frac{\abs{\partial_x (\mathfrak{b}_k p)}^2}{p}\rd x
    \end{equation}
    holds for $\bbP\times \mu$-almost all $(\omega,x)$.
\end{proposition}
\begin{proof}
Applying It\^o's formula to \eqref{eq:sdepartialxgammax} gives
\begin{equation*}
    \rd \log\abs{\partial_x \phi_{\sigma}(t,\omega,x)} = 
    \sigma^2 \mathfrak{a}'(\gamma^x_{\sigma}) \rd t
    - \tfrac{1}{2}\sigma^2 \sum_{k \in \bbZ}\mathfrak{b}_k'(\gamma^x_{\sigma})^2 \rd t
    + \sigma \sum_{k \in \bbZ} \mathfrak{b}_k'(\gamma^x_{\sigma}) \rd \beta_k(t).
\end{equation*}
Integrating this and recalling $\gamma^x_{\sigma}(t) = \phi_{\sigma}(t,\omega,x)$, we get $\lambda = \lim_{t \to \infty}\lambda_1(t)+\lambda_2(t)$ with
\begin{align*}
    \lambda_1(t) &\coloneq \sigma^2 t^{-1}
    \int_0^t \mathfrak{a}'(\gamma^x_{\sigma}(s)) - \tfrac{1}{2}\sum_{k \in \bbZ}\mathfrak{b}_k'(\gamma^x_{\sigma}(s))^2 \rd s, \\
    \lambda_2(t) &\coloneq \sigma t^{-1} \sum_{k \in \bbZ} \int_0^t \mathfrak{b}_k'(\gamma^x_{\sigma}(s)) \rd \beta_k(s).
\end{align*}
Since $\norm{b}_{\ell^2(\bbZ;C^1(\bbT))} < \infty$ by \cref{prop:frakabreg2}, it follows from the strong law of large numbers for martingales that $\lim_{t \to \infty} \lambda_2(t) = 0$ almost surely for every $x \in \bbT$.
For $\lambda_1$, it follows from \cref{prop:mixing} and Birkhoff's pointwise ergodic theorem that 
\begin{equation*}
    \lim_{t \to \infty}\lambda_1(t) = \sigma^2 \int_{\bbT}\mathfrak{a}'(x) -\tfrac{1}{2}\sum_{k \in \bbZ}(\mathfrak{b}_k'(x))^2 \rd \mu(x)
\end{equation*}
for $\bbP \times \mu$-almost every $(\omega,x)$, so the first identity in \eqref{eq:sdelyapunov} holds.

For the second identity, we follow the proof of \cite[Proposition 3.2]{bedrossian_regularity_2022}.
Recall from \cref{prop:mixing} that $\mu = p \rd x$ where $p$ satisfies the Fokker--Planck equation $\cL^*p = 0$.
Multiplying \eqref{eq:FPgenerator} by $\sigma^{-2}\log(p)$ and integrating, we find:
\begin{equation}
    \label{eq:FPtested}
    \tfrac{1}{2}\sum_{k \in \bbZ}\int_{\bbT}\partial_{xx}(\mathfrak{b}_k^2\, p)\log(p)\rd x
    = \int_{\bbT}\partial_x ((\sigma^{-2}c + \mathfrak{a})p)\log(p) \rd x
    = \int_{\bbT}(\partial_x \mathfrak{a}) p \rd x,
\end{equation}
where the second identity in \eqref{eq:FPtested} is obtained by integrating by parts back and forth.
For the left-hand side of \eqref{eq:FPtested}, we additionally find:
\begin{align}
    \tfrac{1}{2}\sum_{k \in \bbZ}\int_{\bbT}\partial_{xx}(\mathfrak{b}_k^2\, p)\log(p)\rd x
    &= -\tfrac{1}{2}\sum_{k \in \bbZ}\int_{\bbT}\frac{\partial_x(\mathfrak{b}_k^2 \,p)\partial_x p}{p}\rd x \nonumber \\ 
    &= \tfrac{1}{2}\sum_{k \in \bbZ}\int_{\bbT}\frac{p^2 (\partial_x \mathfrak{b}_k)^2 - (\partial_x (\mathfrak{b}_kp))^2}{p}\rd x \nonumber \\
    \label{eq:FPtestedLHS}
    &= \tfrac{1}{2}\sum_{k \in \bbZ}\int_{\bbT} (\partial_x \mathfrak{b}_k)^2 p - \frac{(\partial_x (\mathfrak{b}_k p))^2}{p}\rd x.
\end{align}
Combining \eqref{eq:FPtested} and \eqref{eq:FPtestedLHS}, the second identity in \eqref{eq:sdelyapunov} follows.
\end{proof}

\begin{corollary}
    \label{cor:asymstable}
    We have $\lambda < 0$, and $\phi_{\sigma}$ is asymptotically stable in the sense of \cite[Definition 2.2]{flandoli_synchronization_2017}.
\end{corollary}
\begin{proof}
    Since $\mu = p\rd x$ is a probability measure, we see from \eqref{eq:sdelyapunov} and Jensen's inequality that
    \begin{equation*}
        \lambda = -\tfrac{1}{2}\sigma^2 \sum_{k \in \bbZ} \int_{\bbT}\frac{\abs{\partial_x (\mathfrak{b}_k p)}^2}{p^2}p\rd x
        \leq -\tfrac{1}{2}\sigma^2\sum_{k \in \bbZ}\int_{\bbT}\abs{\partial_x(\mathfrak{b}_k p)} \rd x \leq 0.
    \end{equation*}
    Furthermore, if $\lambda = 0$ then we would necessarily have
    \begin{equation*}
        \partial_x (\mathfrak{b}_1p) \equiv \partial_x (\mathfrak{b}_{-1}p) \equiv 0,
    \end{equation*}
    which (since $p > 0$) would further imply that $\mathfrak{b}_1(\cdot)$ and $\mathfrak{b}_{-1}(\cdot)$ are linearly dependent.
    However, by \cref{lem:sinspan} this is not the case so we have $\lambda < 0$ by contraposition.
    From \cref{prop:rds}-\ref{it:rds:integrable} and Jensen's inequality we also see that \cite[(3.3)-(3.4)]{flandoli_synchronization_2017} are satisfied, so that asymptotic stability holds by \cite[Corollary 3.4]{flandoli_synchronization_2017} (which extends \emph{mutatis mutandis} to the torus).
\end{proof}

\subsection{Irreducibility and controllability}
\label{subsec:sdecontrol}
We now verify two irreducibility-type conditions needed for weak synchronization, which are formulated in terms of lower bounds on the probabilities of certain events involving $\phi_{\sigma}$ (see \cref{prop:pwsst,prop:meeting}).
Since our noise is multiplicative, this is not a trivial matter.
Our strategy is to first study a control system associated with the two-point motion of \eqref{eq:sdegammasimp}, where $\diff \beta_1$ and $\diff \beta_{-1}$ are replaced with appropriate control functions.
Once the relevant controllability properties are shown, the irreducibility conditions follow by applying the support theorem for diffusions by \textcite{stroock_support_1972} (see also \cite[Theorem 3.5]{millet_simple_1994} for a version which suffices for our purposes).

Throughout this section we write $\cH$ for the space of piecewise constant compactly supported functions from $\bbR^+$ to $\bbR^2$ and write $h = (h_1,h_{-1})$ for $h \in \cH$.
The functions $h \in \cH$ will serve as controls.
For $x \in \bbR$ and $h \in \cH$ we write, in analogy with our previous notation, $\tilde{\gamma}^x_h(t)$ and $\gamma^x_h(t)$ for the unique solutions to the following ODEs:
\begin{subequations}
\begin{align}
    \label{eq:odegammatildehx}
    \frac{\diff \tilde{\gamma}^x_h}{\diff t} &= \sigma \bra[\big]{\mathfrak{b}_{1}(\tilde{\gamma}^x_h)h_1 + \mathfrak{b}_{-1}(\tilde{\gamma}^x_h)h_{-1}}, \\
    \label{eq:odegammahx}
    \frac{\diff \gamma^x_h}{\diff t} &= c 
    + \sigma^2\bra[\big]{ \mathfrak{a}(\gamma^x_h) - \tfrac{1}{2}\sum_{k \in \bbZ}\mathfrak{b}_k'(\gamma^x_h)\mathfrak{b}(\gamma^x_h)}
    + \sigma \bra[\big]{\mathfrak{b}_{1}(\gamma^x_h)h_1 + \mathfrak{b}_{-1}(\gamma^x_h)h_{-1}},
\end{align}
\end{subequations}
with initial conditions $\tilde{\gamma}^x_h(0) = \gamma^x_h(0) = x$.
Note that \eqref{eq:odegammahx} resembles the SDE for $\gamma_{\sigma}$ in Stratonovich form \eqref{eq:sdegammasimpstrat}, except we have replaced the driving white noises $\rd \beta_{k}$ by control functions $h_{k}$, all of which are zero except for $k = \pm 1$.

\subsubsection{Pointwise strong swift transitivity}
The first irreducibility-type condition that we show is pointwise strong swift transitivity, introduced in \cite[Definition 2.22]{flandoli_synchronization_2017}.
\begin{lemma}
    \label{lem:psstcontrolgammahx}
    For all $x_1,x_2,y \in \bbT$ with $x_1 \neq x_2$, there exists $h \in \cH$ such that
    \begin{equation*}
        \max\cur{d(\gamma^{x_1}_h(1),y),d(\gamma^{x_2}_h(1),y)} \leq \tfrac{3}{2}d(x_1,x_2). 
    \end{equation*}
\end{lemma}
\begin{proof}
    Since $x_1 \neq x_2$ we may write $a = d(x_1,x_2) > 0$.
    Now let $\eta \in \bbT$ be such that $d(\eta,y) \leq a$ and $\min\cur{d(x_1,\eta+\tfrac{1}{2}),d(x_2,\eta+\tfrac{1}{2})} \geq a/4$.
    By \cref{lem:sinspan} we can find for any $L > 0$ a control $h \in \cH$ such that
    \begin{equation*}
        \frac{\diff \gamma^{x}_h}{\diff t} =  c 
    + \sigma^2\bra[\big]{ \mathfrak{a}(\gamma^{x}_h) - \tfrac{1}{2}\sum_{k \in \bbZ}\mathfrak{b}_k'(\gamma^{x}_h)\mathfrak{b}(\gamma^{x}_h)} - L\sin(2\pi \gamma^{x}_h - \eta)
    \end{equation*}
    for $x \in \bbR$ and $t \in [0,1]$.
    Examining the sign of the right-hand side, we see that by choosing $L$ sufficiently large (depending on $c$, $\sigma$, $\mathfrak{a}$, $\mathfrak{b}$, $x_1$, $x_2$, $y$) we can ensure that $\max\cur{d(\gamma^{x_1}_h(1),\eta),d(\gamma^{x_2}_h(1),\eta)} \leq a/2$.
    The claim then follows by the triangle inequality since $d(\eta,y) \leq a = d(x_1,x_2)$.
\end{proof}
\begin{proposition}
    \label{prop:pwsst}
    The RDS $\phi_{\sigma}$ (on $\bbT$) is pointwise strongly swift transitive in the sense of \cite[Definition 2.22]{flandoli_synchronization_2017}.
\end{proposition}
\begin{proof}
    Fix $x_1,x_2,y \in \bbT$ with $x_1 \neq x_2$\footnote{Note that \cite[Definition 2.22]{flandoli_synchronization_2017} should contain the additional condition $x_1 \neq x_2$, as confirmed to us by the authors.}.
    Applying the support theorem for diffusions to the two-point motion $(\phi_{\sigma}(t,\omega,x_1), \phi_{\sigma}(t,\omega,x_2))$ and using \cref{lem:psstcontrolgammahx}, the claim follows.
\end{proof}

\subsubsection{Meeting condition}
The second irreducibility-type condition needed is that for any $x_1, x_2 \in \bbT$, we must have:
\begin{equation}
    \label{eq:meeting}
    \PP[\Big]{\liminf_{t \to \infty} d(\phi_{\sigma}(t,\omega,x_1),\phi_{\sigma}(t,\omega,x_2)) = 0} = 1
\end{equation}
(see \cite[(2.10)]{flandoli_synchronization_2017}).
We refer to \eqref{eq:meeting} (which is unnamed in \cite{flandoli_synchronization_2017}) as the \emph{meeting condition}, since it requires that any two trajectories meet arbitrarily closely infinitely often.
To show \eqref{eq:meeting}, our strategy is to consider for $\eps,T > 0$ the events
\begin{equation}
    \label{eq:defmeetingevent}
    B_n \coloneq \cur[\bigg]{\sup_{x_1,x_2 \in \bbT}\inf_{t \in [0,5T]} d(\phi_{\sigma}(t,\theta_{5nT}\omega,x_1),\phi_{\sigma}(t,\theta_{5nT}\omega,x_2)) \leq 4\eps}, \quad n \in \bbN_0,
\end{equation}
and make the following observations:
\begin{itemize}
    \item $B_n$ is $\cF_{5nT,5(n+1)T}$-measurable by \cref{prop:rds}-\ref{it:rds:white}, and thus the events $B_n$ are independent.
    \item $\PP{B_n} = \PP{B_0}$ for all $n \in \bbN$ by invariance of $\bbP$ under $\theta$.
    \item $\cur{B_n \text{ occurs infinitely often}} \subset \cur{\liminf_{t \to \infty} d(\phi_{\sigma}(t,\omega,x_1),\phi_{\sigma}(t,\omega,x_2)) \leq 4\eps}$.
\end{itemize}
Thus, if we can show that for every $\eps > 0$ there exists a $T > 0$ such that $\PP{B_0} > 0$, then \eqref{eq:meeting} will follow by an application of the second Borel--Cantelli lemma.

Showing $\PP{B_0} > 0$ is still challenging, since $B_0$ involves all trajectories simultaneously.
Our trick to circumvent this is to show that $B_0$ is implied by a certain event $A$ which involves only three trajectories.
To illustrate, consider the `squeezing' event depicted in \cref{fig:squeezing}, involving trajectories of \eqref{eq:sdegammasimp} starting from $z_1 = 0$, $z_2 = 1/3$, $z_3 = -1/3$.
In this event, any two trajectories starting in between $z_2$ and $z_3$ on the same side as $z_1$ will be squeezed together near $z_1$ by monotonicity.
By performing such a squeezing consecutively first near $z_1$, then $z_2$, and finally $z_3$, it follows by the pigeonhole principle that \emph{any} two points must get squeezed at some point of the cycle, which guarantees that the event $B_0$ occurs.

It then only remains to show that such a `triple squeezing cycle' occurs with positive probability.
Since this only involves the three-point motion $(\phi_{\sigma}(t,\omega,z_i))_{i \in \cur{1,2,3}}$ this can be done similarly to how we showed pointwise strong swift transitivity.

\begin{figure}
    \centering
    \includegraphics[width=0.8\textwidth]{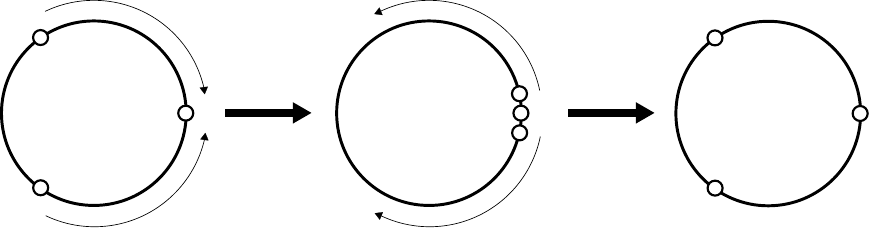}
    \caption{Trajectories starting at $z_2$ and $z_3$ `squeeze' near $z_1$ and afterwards return to their initial position.}
    \label{fig:squeezing}
\end{figure}

Let us now make these ideas rigorous.
We begin by exhibiting a control for \eqref{eq:odegammatildehx} which witnesses the aforementioned squeezing cycle for the three-point motion.
\begin{lemma}
    \label{lem:kissingcontrolgammatildehx}
    Let $z_1 = 0$, $z_2=1/3$, $z_3 = -1/3$.
    For every $\eps > 0$ there exists $T>0$, $h \in \cH$ such that
    \begin{equation}
    \label{eq:kissingcontrolgammatildehx}
    \begin{aligned}
        z_1-\eps \leq \tilde{\gamma}^{z_3}_h(T)&< \tilde{\gamma}^{z_2}_h(T) \leq z_1+\eps, \\
        z_2-\eps \leq \tilde{\gamma}^{z_1}_h(3T)&< \tilde{\gamma}^{z_3}_h(3T)+1 \leq z_2+\eps, \\
        z_3-\eps \leq \tilde{\gamma}^{z_2}_h(5T)-1&< \tilde{\gamma}^{z_1}_h(5T) \leq z_3+\eps.
    \end{aligned}
    \end{equation}
\end{lemma}

\begin{proof}
    We show that $T = 1$ suffices.
    Using \eqref{eq:sinspan} we can find for any $L > 0$ a control $h \in \cH$ such that
    \begin{equation}
        \label{eq:odegammatildehxkissing}
        \frac{\diff \tilde{\gamma}^x_h}{\diff t} = -k(t)L \sin(2 \pi \tilde{\gamma}^x_h(t) - \eta(t)), \qquad t \geq 0,\, x \in \bbR,
    \end{equation}
    where
    \begin{align*}
        k(t) &= 
          \ind_{[0,1]}(t)
        - \ind_{[1,2]}(t)
        + \ind_{[2,3]}(t)
        - \ind_{[3,4]}(t)
        + \ind_{[4,5]}(t), \\
        \eta(t) &= z_1 \cdot \ind_{[0,2]}(t) + z_2 \cdot\ind_{[2,4]}(t) 
                 + z_3 \cdot \ind_{[4,5]}(t) .
    \end{align*}
    With $h$ of this form, it follows from time reversal of \eqref{eq:odegammatildehx} that $x = \tilde{\gamma}^x_h(2) = \tilde{\gamma}^x_h(4)$ for every $x \in \bbR$.
    After examining \eqref{eq:odegammatildehxkissing} on the time intervals $[0,1]$, $[2,3]$, and $[4,5]$, it is seen that \eqref{eq:kissingcontrolgammatildehx} holds when $L$ is sufficiently large (depending on $\sigma$ and $\mathfrak{b}_{\pm1}$).
\end{proof}
By a perturbative argument we now show that the same statement holds for \eqref{eq:odegammahx}.
\begin{lemma}
    \label{lem:kissingcontrolgammahx}
    \cref{lem:kissingcontrolgammatildehx} still holds when $\tilde{\gamma}$ is replaced by $\gamma$ in \eqref{eq:kissingcontrolgammatildehx}.
\end{lemma}
\begin{proof}
    Fix $\eps > 0$ and let $T>0,h\in\cH$ be such that \eqref{eq:kissingcontrolgammatildehx} is satisfied.
    For $\delta > 0$ we define $h_\delta$ via $h_\delta(\cdot) = \delta^{-1} h(\delta^{-1}\,\cdot)$.
    By rescaling time in \eqref{eq:odegammatildehx} we see that
    \begin{equation}
        \label{eq:gammatildehxrescale}
        \tilde{\gamma}^x_{h_\delta}(\cdot) = \tilde{\gamma}^x_h(\delta^{-1}\,\cdot), \qquad x \in \bbR,
    \end{equation}
    for any $\delta > 0$.
    Furthermore, subtracting \eqref{eq:odegammatildehx} from \eqref{eq:odegammahx} and using \cref{prop:frakabreg2} we find that there exists $C > 0$ (depending only on $\sigma$, $c$, $\mathfrak{a}$, $\mathfrak{b}$) such that
    \begin{equation*}
        \abs{\gamma^x_{h_\delta}(t)-\tilde{\gamma}^x_{h_\delta}(t)} 
        \leq Ct + C \delta^{-1} \int_0^t \abs{\gamma^x_{h_\delta}(s)-\tilde{\gamma}^x_{h_\delta}(s)}\rd s, \qquad x \in \bbR,\,t \geq 0,\,\delta > 0.
    \end{equation*}
    It follows from Gr\"onwall's lemma that
    \begin{equation*}
        \sup_{t \in [0,\delta T]} \abs{\gamma^x_{h_\delta}(t),\tilde{\gamma}^x_{h_\delta}(t)} \leq \delta CT\exp(CT), \qquad \delta > 0.
    \end{equation*}
    Choosing $\delta$ sufficiently small (depending only on $\eps,C,T$) and 
    combining this with \eqref{eq:kissingcontrolgammatildehx} and \eqref{eq:gammatildehxrescale}, we see that \eqref{eq:kissingcontrolgammatildehx} is still satisfied with $\tilde{\gamma}$ replaced by $\gamma$, $T$ replaced by $\delta T$, $h$ replaced by $h_{\delta}$, and $\eps$ replaced by $2\eps$.
    The conclusion follows since $\eps$ was arbitrary.
\end{proof}

Finally, we show that $\PP{B_0}$ has positive probability using monotonicity of \eqref{eq:sdegammasimp} and the Stroock--Varadhan support theorem.

\begin{lemma}
    \label{lem:squeezinghappens}
    Let $\eps > 0$ and let $T$ be as in \cref{lem:kissingcontrolgammahx}.
    Then $\PP{B_0} > 0$.
\end{lemma}
\begin{proof}
    Applying the support theorem to the three-point motion $(\phi_{\sigma}(t,\omega,z_i))_{i \in \cur{1,2,3}}$ and using \cref{lem:kissingcontrolgammahx}, we see that there is a strictly positive probability that \eqref{eq:kissingcontrolgammatildehx} holds with $\tilde{\gamma}^{z_i}_h(T)$ replaced by $\phi_{\sigma}(T,\omega,z_i)$ everywhere.
    By monotonicity it then follows that the event
    \begin{align*}
        A \coloneq &\cur[\big]{\phi_{\sigma}(T,\omega,[z_3,z_2]) \subset B_{2\eps}(z_1)}
        \cap \cur[\big]{\phi_{\sigma}(3T,\omega,[z_1,z_3+1]) \subset B_{2\eps}(z_2)} \\
        &\qquad\cap \cur[\big]{\phi_{\sigma}(5T,\omega,[z_2-1,z_1]) \subset B_{2\eps}(z_3)}
    \end{align*}
    (where $\phi_{\sigma}$ is interpreted as an RDS on $\bbR$)
    has positive probability.
    Now observe that for any two-point set $S \subset \bbT$ it must be the case that either $S \subset [z_3,z_2]$, $S \subset [z_1,z_3+1]$, or $S \subset [z_2-1,z_1]$ (for properly chosen representatives of $S$) by the pigeonhole principle.
    Thus, the event $A$ implies $B_0$ so that $\PP{B_0} \geq \PP{A} > 0$.
\end{proof}
\begin{proposition}
    \label{prop:meeting}
    The meeting condition \eqref{eq:meeting} holds true.
\end{proposition}
\begin{proof}
    Fix $\eps > 0$.
    From the second Borel--Cantelli lemma, \cref{lem:squeezinghappens} and the observations directly following \eqref{eq:defmeetingevent} we see that
    \begin{equation*}
        \PP[\Big]{\liminf_{t \to \infty} d(\phi_{\sigma}(t,\omega,x_1),\phi_{\sigma}(t,\omega,x_2)) \leq 4\eps} = \PP{B_n \text{ infinitely often}} = 1.    
    \end{equation*}
    The claim follows since $\eps$ was arbitrary.
\end{proof}

\subsection{Uniform weak synchronization}
\label{subsec:sdesync}
From this point onward we shall exclusively view $\phi_{\sigma}$ as an RDS on $\bbT$.
\begin{theorem}
    \label{thm:phisigmasync}
    The RDS $\phi_{\sigma}$ synchronizes weakly, i.e., there is a minimal weak point attractor $a(\omega)$ consisting of a single point.
    Moreover, for any $\eps > 0$ we have
    \begin{equation}
        \label{eq:uniformsync}
        \lim_{t \to \infty}\sup_{x\in \bbT} \PP[\big]{d(\phi_{\sigma}(t,\omega,x),a(\theta_t\omega)) > \eps} = 0,
    \end{equation}
    i.e., the synchronization rate is uniform in the initial condition.
\end{theorem}
\begin{proof}
    Taking \cref{prop:rds,prop:mixing,cor:asymstable,prop:pwsst,prop:meeting} together, we see that all the conditions of \cite[Theorem 2.23]{flandoli_synchronization_2017} are satisfied and weak synchronization holds.
    
    To show uniform synchronization we will additionally use that the strong mixing rate of $\phi_{\sigma}$ is uniform in the initial condition (see \cref{prop:mixing,rem:TV}).
    Fix $\eps,\delta > 0$ and let $z \in \bbT$ be arbitrary.
    Weak synchronization implies
    \begin{equation*}
        \lim_{t \to \infty}\PP{d(\phi_{\sigma}(t,\omega,y),\phi(t,\omega,z)) > \eps} = 0, \quad y \in \bbT,
    \end{equation*}
    so by dominated convergence we can find a time $t_2 > 0$ such that
    \begin{align}
        \label{eq:unifsyncPmu}
        \int_{\bbT} \PP{d(\phi_{\sigma}(t_2,\omega,y),\phi_{\sigma}(t_2,\omega,z)) > \eps} \rd \mu(y) \leq \delta, \\
        \label{eq:unifsynczatheta}
        \PP{d(\phi(t_2,\omega,z),a(\theta_{t_2}\omega)) > \eps} \leq \delta.
    \end{align}
    Using the auxiliary function $\psi\colon x \mapsto (x-1)\ind_{[1,2]}(x) + \ind_{(2,\infty)}(x)$ we now define
    \begin{equation*}
        f(x) = \EE{\psi(\eps^{-1}d(\phi_{\sigma}(t_2,\omega,x),\phi_{\sigma}(t_2,\omega,z)))}.
    \end{equation*}
    Notice that $f \in C_b(\bbT)$ by continuity of $\phi_{\sigma},\psi$ and dominated convergence, and also that
    \begin{equation}
        \label{eq:unifsyncfbound}
     \PP{d(\phi_{\sigma}(t_2,\omega,x),\phi_{\sigma}(t_2,\omega,z)) > 2\eps} \leq f(x) \leq \PP{d(\phi_{\sigma}(t_2,\omega,x),\phi_{\sigma}(t_2,\omega,z)) > \eps}
    \end{equation}
    for every $x \in \bbT$, which implies $\int_{\bbT}f \rd \mu \leq \delta$ by \eqref{eq:unifsyncPmu}.
    By \eqref{eq:uniformmixing} we then find a time $t_1 > 0$ such that
    \begin{equation}
        \label{eq:unifsyncPtbound}
        \sup_{x \in \bbT} \abs{P_t f(x)} \leq \delta + \int_{\bbT} f \rd \mu
        \leq 2 \delta,\qquad t \geq t_1.
    \end{equation}
    Now for $t = s + t_2$ with $s \geq t_1$ we find by the cocycle property and invariance of $\bbP$:
    \begin{align*}
        &\PP[\big]{d(\phi_{\sigma}(s+t_2,\omega,x),\phi_{\sigma}(t_2,\theta_{s}\omega,z)) > 2\eps}
        =\PP[\big]{d(\phi_{\sigma}(t_2,\omega,\phi_{\sigma}(s,\theta_{-s}\omega,x)),\phi_{\sigma}(t_2,\omega,z)) > 2\eps} \\
        &\qquad= P_s\bra[\Big]{ \PP[\big]{d(\phi_{\sigma}(t_2,\omega,\cdot),\phi_{\sigma}(t_2,\omega,z)) > 2\eps}}(x)
        \overset{\eqref{eq:unifsyncfbound}}{\leq} \sup_{y \in \bbT} \abs{P_s f(y)} \overset{\eqref{eq:unifsyncPtbound}}{\leq} 2 \delta, \qquad x \in \bbT,
    \end{align*}
    where the second identity follows since $\phi(s,\theta_{-s}\omega,x)$ is independent of $\phi_{\sigma}(t_2,\omega,\cdot)$ by \cref{prop:rds}-\ref{it:rds:white}.
    Combining this with the triangle inequality and using the invariance of $\bbP$ again yields
    \begin{align*}
        &\PP[\big]{d(\phi_{\sigma}(s+t_2,\omega,x),a(\theta_{s+t_2}\omega)) > 3\eps} 
        \leq 2\delta + \PP[\big]{d(\phi_{\sigma}(t_2,\theta_{s}\omega,z),a(\theta_{t_2+s}\omega)) > \eps} \\
        &\qquad\qquad= 2\delta + \PP[\big]{d(\phi_{\sigma}(t_2,\omega,z),a(\theta_{t_2}\omega)) > \eps}
        \overset{\eqref{eq:unifsynczatheta}}{\leq} 3\delta
    \end{align*}
    for all $s \geq t_1$ and $x \in \bbT$.
    Since $t_1$, $t_2$ were chosen independently of $x$, \eqref{eq:uniformsync} follows.
\end{proof}

To prove the main result, it only remains to transfer the synchronization properties of $\phi_{\sigma}$ back to $u_{\sigma}$.
However, here we encounter a limitation of \cref{thm:phisigmasync}, namely that \eqref{eq:uniformsync} does not provide a quantitative lower bound on the synchronization time in terms of $\sigma$.
This poses a problem, since \cref{thm:longtermapprox} shows that the validity of the approximation $\iso(u_{\sigma}^x(t)) \approx \phi_{\sigma}(t,\omega,x)$ is only valid on a time scale $t \sim \sigma^{-2}\log(\sigma^{-1})$.
Hence, if the time until synchronization of $\phi_{\sigma}$ diverges quicker than this as $\sigma \searrow 0$, the synchronization behavior cannot be transferred to $u_{\sigma}$.
Of course, the characteristic time scale of \eqref{eq:sdegammasimp} is $t \sim \sigma^{-2}$, so it is reasonable to expect that the synchronization should take place before the approximation breaks down.
However, since the term $c \rd t$ in \eqref{eq:sdegammasimp} breaks the time-scaling symmetry, we are currently not able to prove this without imposing \cref{ass:scaling} (note that this assumption has been unused until now).
We believe that a quantitative version of \cite[Theorem 2.23]{flandoli_synchronization_2017} could mitigate the need for \cref{ass:scaling}, and leave this as a suggestion for future work.

To overcome the difficulties outlined above, we relate \eqref{eq:sdegammasimp} to the following SDE in which $\sigma$ no longer appears:
\begin{equation}
    \label{eq:sdegammatilde}
    \rd \tilde{\gamma} = \mathfrak{a}(\tilde{\gamma})\rd t + \sum_{k \in \bbZ}\mathfrak{b}_k(\tilde{\gamma}) \rd \beta_k(t).
\end{equation}
Note that if $c = 0$, \eqref{eq:sdegammatilde} can be directly obtained from \eqref{eq:sdegammasimp} by rescaling time.
In the case $c \neq 0$ an additional step is needed.
Throughout the following, we let $\tilde{\phi}$ denote the random dynamical system over $(\theta_t)_{t \in \bbR}$ generated by \eqref{eq:sdegammatilde}.
The following proposition gives an exact relation between $\phi_{\sigma}$ and $\tilde{\phi}$.

\begin{proposition}
    \label{prop:gammarescale}
    For every $\sigma > 0$, there is a map $T \colon \Omega \to \Omega$ which satisfies $T^*\bbP = \bbP$ as well as
    \begin{equation}
        \label{eq:rdsrescale}
        \PP[\Big]{\phi_{\sigma}(t,\omega,x) = ct + \tilde{\phi}(\sigma^2t,T(\omega),x)} = 1, \qquad t \in \bbR^+,\,x \in \bbT.
    \end{equation}
\end{proposition}
\begin{proof}
    Consider first the case $c = 0$.
    We set $\tilde{\beta}_k(\cdot) = \sigma \beta_k(\sigma^{-2} \cdot)$ and let
    $T_1$ be the associated transformation of $\Omega$.
    It is a basic property of Brownian motion that $T_1^*\bbP = \bbP$.
    By rescaling time in \eqref{eq:sdegammasimp} it is also seen that $\tilde{\gamma}(t) \coloneq \gamma(\sigma^{-2}t)$ solves \eqref{eq:sdegammatilde} with $\beta_k$ replaced by $\tilde{\beta}_k$, so that \eqref{eq:rdsrescale} holds.

    Consider now the case $c \neq 0$. 
    By \cref{ass:scaling} we must have $\alpha_k = \alpha_{-k}$ for every $k \in \bbN$.
    It then follows from \eqref{eq:frakabtiv} that $\mathfrak{a}(x)$ is constant, so we simply write $\mathfrak{a}$.
    Making the substitution $\tilde{\gamma}(t) \coloneq \gamma(t) - ct$, we combine \eqref{eq:sdegammasimp} and \eqref{eq:simplenoise} to find:
    \begin{align}
        \label{eq:sdegammatildesimplenoise}
        \diff \tilde{\gamma}
        &=\sigma^2 \mathfrak{a} \rd t + \sigma \bra[\bigg]{\mathfrak{b}_0 \rd \beta_0(t) +
            \sum_{k \in \bbN}\alpha_k\begin{pmatrix}
                c_k \\ c_{-k}
            \end{pmatrix}^{\top} R(-2\pi k\tilde{\gamma})
            R(-2\pi kct)
            \begin{pmatrix}
                \rd \beta_k \\ \rd \beta_{-k}
            \end{pmatrix}
        }
    \end{align}
    We now set $\tilde{\beta}_0 = \beta_0$, as well as
    \begin{align*}
        \begin{pmatrix}
            \tilde{\beta}_k(\cdot) \\
            \tilde{\beta}_{-k}(\cdot)
        \end{pmatrix} &= \int_0^{\cdot} R(-2\pi kcs)\begin{pmatrix}
            \rd \beta_k(s) \\ \rd \beta_{-k}(s)
        \end{pmatrix},\quad k \in \bbN,
    \end{align*}
    and let $T_2$ be the associated map on $\Omega$.
    It follows from Levy's characterization of the Brownian motion that $T_2^*\bbP = \bbP$.
    Moreover, it is seen that
    \begin{equation*}
        \begin{pmatrix}
            \diff \tilde{\beta}_k(t) \\ \diff \tilde{\beta}_{-k}(t)
        \end{pmatrix}
        = R(-2\pi kct) \begin{pmatrix}
            \diff \beta_k(t) \\ \diff \beta_{-k}(t)
        \end{pmatrix},
    \end{equation*}
    so that \eqref{eq:sdegammatildesimplenoise} becomes
    \begin{align*}
        \diff \tilde{\gamma} &= \sigma^2 \mathfrak{a}\rd t 
        +  \sigma \bra[\bigg]{\mathfrak{b}_0 \rd \tilde{\beta}_0(t) +
            \sum_{k \in \bbN}\alpha_k\begin{pmatrix}
                c_k \\ c_{-k}
            \end{pmatrix}^{\top} R(-2 \pi k\tilde{\gamma})
            \begin{pmatrix}
                \rd \tilde{\beta}_k \\ \rd \tilde{\beta}_{-k}
            \end{pmatrix}
        } \\
        &= \sigma^2 \mathfrak{a} \rd t + \sigma \sum_{k \in \bbZ}\mathfrak{b}_k(\tilde{\gamma}) \rd \tilde{\beta}_k.
    \end{align*}
    Recalling that $\gamma = \tilde{\gamma} + ct$ and using the result with $c=0$, we find a map $T_1 \colon \Omega \to \Omega$ which leaves $\bbP$ invariant and satisfies
    \begin{equation*}
        \phi_{\sigma}(t,\omega,x) = ct + \tilde{\phi}(\sigma^2 t, T_1 \circ T_2(\omega),x).
    \end{equation*}
    We conclude by noting that $(T_1 \circ T_2)^*\bbP = T_1^* (T_2^* \bbP) = \bbP$.
\end{proof}

\subsection{Proof of the main result}
\begin{proof}[Proof of \cref{thm:spdesync}]
    Let $(t_{\sigma})_{\sigma > 0}$ and $q$ be as in the theorem statement and fix $x,y \in \bbR$.
    By \eqref{eq:tsigmacond} and \cref{thm:longtermstab} (taking into account \cref{rem:translateustar}), it suffices to prove
    \begin{equation}
        \label{eq:dpixyconvprob}
        d(\iso(u_{\sigma}^x(t_{\sigma})), \iso(u_{\sigma}^y(t_{\sigma}))) \overset{\bbP}{\to} 0 \quad \text{as} \quad \sigma \to 0.
    \end{equation}
    Also by \eqref{eq:tsigmacond}, we can find a decomposition $t_{\sigma} = s_{\sigma} + c_{\sigma}\sigma^{-2}$ with $s_{\sigma} \geq 0$, and $c_{\sigma}$ satisfying
    \begin{equation}
        \label{eq:csigmacond}
        0 \leq c_{\sigma} \leq \log(\sigma^{-1})^{1-q/9},
        \qquad \lim_{\sigma \to 0} c_{\sigma} = \infty.
    \end{equation}
    It then follows from (a time-shifted version of) \cref{thm:longtermapprox} that
    \begin{equation}
        \label{eq:dpixphiconvprob}
        d(\iso(u_{\sigma}^x(t_{\sigma})), \phi_{\sigma}(c_{\sigma}\sigma^{-2},\theta_{s_{\sigma}}\omega,\iso(u_{\sigma}^x(s_{\sigma})))) \overset{\bbP}{\to} 0 \quad \text{as} \quad \sigma \to 0,
    \end{equation}
    and likewise for $y$.
    Hence, abbreviating $x_{s} = \iso(u^x_{\sigma}(s_{\sigma}))$ (and similarly for $y$) it will suffice to prove
    \begin{equation}
        \label{eq:dphixyconvprob}
        d(\phi_{\sigma}(c_{\sigma}\sigma^{-2},\theta_{s_{\sigma}}\omega,x_s),\phi_{\sigma}(c_{\sigma}\sigma^{-2},\theta_{s_{\sigma}}\omega,y_s)) \overset{\bbP}{\to} 0 \quad \text{as} \quad \sigma \to 0,
    \end{equation}
    since \eqref{eq:dpixyconvprob} then follows from \eqref{eq:dpixphiconvprob}-\eqref{eq:dphixyconvprob} and the triangle inequality.
    Fix now $\eps > 0$.
    Since $x_s$ and $y_s$ are both $\cF_{s_{\sigma}}$-measurable, they are both independent of the map $z \mapsto \phi_{\sigma}(c_{\sigma}\sigma^{-2},\theta_{s_{\sigma}}\omega,z)$ by \cref{prop:rds}-\ref{it:rds:white}.
    Thus, we find
    \begin{align*}
        &\PP{d(\phi_{\sigma}(c_{\sigma}\sigma^{-2},\theta_{s_{\sigma}}\omega,x_s),\phi_{\sigma}(c_{\sigma}\sigma^{-2},\theta_{s_{\sigma}}\omega,y_s)) > \eps}
        \\ &\quad\leq 
        \sup_{x,y \in \bbT}
        \PP{d(\phi_{\sigma}(c_{\sigma}\sigma^{-2},\theta_{s_{\sigma}}\omega,x),\phi_{\sigma}(c_{\sigma}\sigma^{-2},\theta_{s_{\sigma}}\omega,y)) > \eps} \\
        &\quad= \sup_{x,y \in \bbT}
        \PP{d(\phi_{\sigma}(c_{\sigma}\sigma^{-2},\omega,x),\phi_{\sigma}(c_{\sigma}\sigma^{-2},\omega,y)) > \eps} \\
        &\quad= \sup_{x,y \in \bbT}
        \PP{d(\tilde{\phi}(c_{\sigma},\omega,x),\tilde{\phi}(c_{\sigma},\omega,y)) > \eps}
    \end{align*}
    for every $\sigma > 0$,
    where we have used \cref{prop:gammarescale} for the final step.
    Using \eqref{eq:csigmacond} and \cref{thm:phisigmasync} (which applies equally well to $\tilde{\phi}$) we now conclude that \eqref{eq:dphixyconvprob} holds.
\end{proof}

\section{Outlook}
\label{sec:outlook}
We have shown that the phase reduction approach is an effective way to prove synchronization of traveling pulses.
We now briefly discuss some possible extensions of the result, including the possibility to remove \cref{ass:scaling}.

\subsection{Fixed noise amplitude}
Instead of considering the joint limit of small noise and long time as in \cref{thm:spdesync}, one could instead try to show synchronization for a fixed noise amplitude $\sigma > 0$.
In this case, analysis of the long-time behavior is complicated by the fact that the pulse is only metastable, and typically has a limited lifetime.
This may be remedied using the theory of \emph{quasi-ergodic} (i.e., conditioned on survival) measures, which have recently been shown to exist in a setting  similar to ours \cite{adams_quasiergodicity_2024}.
Furthermore, the existence of a conditioned Lyapunov exponent has been shown in \cite{engel_conditioned_2019}, which was recently extended to exhibit a full conditioned Lyapunov spectrum \cite{castro_lyapunov_2022}.
We believe this is an interesting avenue for future research, and expect that further developments of this theory will be helpful to show transient synchronization.

\subsection{Higher-dimensional patterns and chaos}
By the phase reduction method outlined in \cref{sec:phasereduction}, the dynamics of the SPDE \eqref{eq:spde} near the traveling pulse are essentially one-dimensional, and are accurately described by \eqref{eq:sdegammasimp}.
Thus, chaotic behavior is immediately ruled out by monotonicity.
However, the phase reduction method has also been applied to patterns with more degrees of freedom, such as rotating and spiral waves \cite{kuehn_stochastic_2022,vanwinden_noncommutative_2024} and multidimensional traveling waves \cite{vandenbosch_multidimensional_2024}.
In these cases the reduced SDEs are multidimensional, so chaotic behavior might occur as opposed to synchronization.
We expect that the sign of the top Lyapunov exponent (which commonly characterizes chaos/synchronization) could be determined using either analytical \cite{bedrossian_regularity_2022} or computer-assisted techniques \cite{breden_computerassisted_2023,breden_rigorous_2024}.
\subsection{Spatially inhomogeneous noise}
\label{subsec:liftscaling}
In the case where the pulse speed $c$ is nonzero, \cref{ass:scaling} restricts the noise to be statistically spatially homogeneous.
From a physical/symmetry perspective this assumption is not artificial or unreasonable.
Moreover, spatially homogeneous noise is frequently used in this setting; see for example \cite{gnann_solitary_2024,hamster_travelling_2020,westdorp_longtimescale_2024,kilpatrick_stochastic_2015}.
From a mathematical point of view, the extra symmetry simplifies many of the computations in \cref{sec:reducedsdeanalysis}.
Most notably, it results in the following:
    \begin{enumerate}
        \item The coefficient $\mathfrak{a}$ in \eqref{eq:sdegammasimp} does not depend on $x$. This also holds for $\norm{\mathfrak{b}}_{\ell^2(\bbZ)}$ and $\norm{\mathfrak{b}'}_{\ell^2(\bbZ)}$. 
        From \eqref{eq:sdegammasimp}, we may interpret the (deterministic) quantity $c+\sigma^2\mathfrak{a}$ as the stochastically corrected pulse speed (c.f.\ \cite[Theorem 2.3]{giacomin_small_2018}) \cite[\S 2.2]{hamster_stability_2020}  \cite[\S 4.2]{adams_isochronal_2025}.
        \item The invariant measure $\mu$ in \cref{prop:mixing} is the Lebesgue measure on $\bbT$.
        \item The Lyapunov exponent in \eqref{eq:sdelyapunov} satisfies $\lambda = -\frac{1}{2}\sigma^2 \norm{\mathfrak{b}'}_{\ell^2(\bbZ)}^2$.
    \end{enumerate}

We note that the validity of the phase reduction (\cref{thm:longtermstab,thm:longtermapprox}), as well as synchronization of the reduced SDE (\cref{thm:phisigmasync}) are all established without use of \cref{ass:scaling}.
This leads us to conjecture that \cref{ass:scaling} might not be needed for \cref{thm:spdesync} to hold.
However, without \cref{ass:scaling} we are currently unable to get suitable quantitative control of the synchronization rate, which is otherwise provided by \cref{prop:gammarescale}.
Hence, we cannot rule out the possibility that the validity of the approximation $\gamma_{\sigma} \approx \iso(u_{\sigma})$ breaks down before synchronization occurs.
We expect that suitable estimates may be obtained by quantifying the results of \cite{flandoli_synchronization_2017}, a problem which we believe to be of independent interest.
The work \cite{vorkastner_approaching_2020} is a first step in this direction.

\subsection{Applications}

Our work also has potential implications for applications, e.g., for the biophysical systems mentioned in~\cref{sec:intro}. For example, spatial synchronization by noise of nerve impulses along axons could be a natural biological robustness mechanism to avoid multiple short-time shifted pulses to arrive at a single neuron within a short time span. Furthermore, our proof also reveals two important aspects: (I) there is only an intermediate synchronization window of a time scale $t_{\sigma}$ with $\sigma^{-2} \ll t_{\sigma} \ll \exp(\sigma^{-2})$, and (II) spatially homogeneous noise potentially could be beneficial for synchronization as suggested by~\cref{ass:scaling}. At first, (I)-(II) seem counter-intuitive for certain applications, e.g., for neuronal dynamics. Yet, (I) can be desirable for controlling pulse dynamics to a critical state, similarly to other self-organized criticality mechanisms, as small perturbations of the time scale can lead to different information processing outcomes. (II) might reveal a connection to chemical and electrical control of action potential propagation, i.e., changing from a heterogeneous chemical or electrical potential around the axon to a more homogeneous one may increase the propensity to synchronize slightly separated pulses along axons.

\printbibliography

\end{document}